\documentclass[11pt]{amsart}

\usepackage[margin=2.7cm]{geometry}

\usepackage[table]{xcolor}

\usepackage{comment}
\usepackage{stmaryrd}

\usepackage{ytableau,tikz,varwidth, hyperref, tikz-cd}
\usetikzlibrary{calc, snakes}
\usepackage[enableskew]{youngtab}
\usepackage{geometry}
\usepackage{enumerate}
\usepackage{amssymb, mathscinet}

\usepackage[all,cmtip]{xy}

\usepackage{mathbbol}

\DeclareSymbolFontAlphabet{\mathbbm}{bbold}
\DeclareSymbolFontAlphabet{\mathbb}{AMSb}

\newtheorem{thm}{Theorem}[section]
\newtheorem{prop}[thm]{Proposition}

\def\J{\mathbbm{\Gamma}}

\newcommand{\ZZ}{\mathbb{Z}}
\newcommand{\RR}{\mathbb{R}}

\newcommand{\QQ}{\mathbb{Q}}
\newcommand{\NN}{\mathbb{N}}

\newcommand{\cM}{\mathcal{M}}

\newcommand{\cQ}{\mathcal{Q}}

\newcommand{\on}{\operatorname}
\newcommand{\Gr}{\operatorname{Gr}}

\newcommand{\ord}{\on{ord}}

\newcommand{\double}{\genfrac..{0pt}1
{\raise -2pt\hbox{$\scriptstyle\longrightarrow$}}{\raise 4pt\hbox
{$\scriptstyle\longrightarrow$}}}

\newcommand{\sgn}{\operatorname{sgn}}

\newcommand{\Kgn}{K^{(g,n)}}  

\newcommand{\sss}{\alpha}
\newcommand{\ttt}{\beta}
\newcommand{\uuu}{\gamma}

\newcommand{\hide}[1]{}

\newtheorem{Definition}[thm]{Definition}
\newenvironment{definition}
  {\begin{Definition}}{\end{Definition}}

\newtheorem{Example}[thm]{Example}
\newenvironment{example}
  {\begin{Example}\rm}{\end{Example}}

\newtheorem{Notation}[thm]{Notation}

\newtheorem{Fact}[thm]{Fact}
\newenvironment{fact}
  {\begin{Fact}\rm}{\end{Fact}}

\newtheorem{Theorem}[thm]{Theorem}
\newenvironment{theorem}
  {\begin{Theorem}}{\end{Theorem}}
  
\newtheorem{Lemma}[thm]{Lemma}
\newenvironment{lemma}
  {\begin{Lemma}}{\end{Lemma}}

\newtheorem{Remark}[thm]{Remark}
\newenvironment{remark}
  {\begin{Remark}\rm}{\end{Remark}}

\newtheorem{Proposition}[thm]{Proposition}
\newenvironment{proposition}
  {\begin{Proposition}}{\end{Proposition}}

\newtheorem{Corollary}[thm]{Corollary}
\newenvironment{corollary}
  {\begin{Corollary}}{\end{Corollary}}

\newtheorem{Question}[thm]{Question}
\newenvironment{question}
  {\begin{Question}\rm}{\end{Question}}

\newtheorem{Conjecture}[thm]{Conjecture}
\newenvironment{conjecture}
  {\begin{Conjecture}\rm}{\end{Conjecture}}

\newcommand \defnow[1]{\begin{definition}{#1}\end{definition}}
\newcommand \exnow[1]{\begin{example}{#1}\end{example}}
\newcommand \lemnow[1]{\begin{lemma}{#1}\end{lemma}}

\newcommand \proofnow[1]{\begin{proof}{#1}\end{proof}}

\newcommand \propnow[1]{\begin{proposition}{#1}\end{proposition}}
\newcommand \cornow[1]{\begin{corollary}{#1}\end{corollary}}

\newcommand \enumnow[1]{\begin{enumerate}{#1}\end{enumerate}}

\newcommand{\Aut}{\on{Aut}}

  \theoremstyle{definition} 

\theoremstyle{remark}

\title{The $S_n$-equivariant top weight Euler characteristic of $\cM_{g,n}$}
\author[M. Chan]{Melody Chan}\address{Department of Mathematics, Brown University, Box
1917, Providence, RI 02912, USA}\email{melody\_chan@brown.edu}
\author[C. Faber]{Carel Faber}\address{Department of Mathematics, Utrecht University, PO Box 80010, 3508 TA Utrecht, The Netherlands}\email{C.F.Faber@uu.nl}
\author[S. Galatius]{S{\o}ren Galatius} \address{Department of Mathematical Sciences, University of Copenhagen, Universitetsparken 5, DK-2100 Copenhagen, Denmark}\email{galatius@math.ku.dk}
\author[S. Payne]{Sam Payne}\address{UT Department of Mathematics, 2515 Speedway, RLM 8.100, Austin, TX 78712, USA}\email{sampayne@utexas.edu}
\date{\today}

\begin{document}

\begin{abstract}
We prove a formula, conjectured by Zagier, for the $S_n$-equivariant Euler characteristic of the top weight cohomology of $\cM_{g,n}$.
\end{abstract}

\maketitle

\section{Introduction}

Fix an integer $g\ge 2$.  For each $n\ge 0$, let $\cM_{g,n}$ denote the moduli space of complex algebraic curves of genus $g$ with $n$ distinct marked points.  The cohomology $H^*(\cM_{g,n};\QQ)$ carries a weight filtration supported in degrees from $0$ to $2d$, where $d = 3g-3+n$ is the dimension of $\cM_{g,n}$.  The group $S_n$ acts by permuting the marked points.  Our main result is a formula for the $S_n$-equivariant Euler characteristic of the induced action on the top weight cohomology $\Gr_{2d}^W H^*(\cM_{g,n};\QQ)$.  It may be stated as follows.

For any partition $\lambda \vdash n$, we have the corresponding irreducible representation $V_\lambda$ of $S_n$, and the Schur function $s_\lambda$ in the ring of symmetric functions $\Lambda = \varprojlim_n \QQ[x_1, \ldots, x_n]^{S_n}$.  Decompose the top weight cohomology of $\cM_{g,n}$ into irreducible representations as
\begin{equation*}
\Gr_{2d}^W H^i(\cM_{g,n};\QQ) \cong \bigoplus_{\lambda\,\vdash\, n} c^{i}_\lambda\,V_\lambda.
\end{equation*}
Then the generating function for the $S_n$-equivariant top weight Euler characteristic of $\cM_{g,n}$ in the ring of formal symmetric power series $\widehat \Lambda = \varprojlim_n \QQ\llbracket x_1, \dots, x_n\rrbracket^{S_n}$ is
\begin{equation*}
z_g = \sum_{i, \lambda} (-1)^i c^i_\lambda s_\lambda;
\end{equation*}
it simultaneously encodes the integers $\sum_i (-1)^i c^i_\lambda$ for all $n$ and all $\lambda \vdash n$.

This formula for $z_g$ is most conveniently expressed in terms of the inhomogeneous power sum functions $P_i = 1 + p_i$, where $p_i = \sum_j x_j^i$.  Let $B_r \in \QQ$ denote the Bernoulli number, characterized by $\frac{t}{e^t - 1} = \sum_{r=0}^\infty B_r \frac{t^r}{r!}$, and let $\mu$ be the M\"obius function.

\begin{theorem}\label{thm:faber-conj}
The $S_n$-equivariant top weight Euler characteristic of $\cM_{g,n}$ is
\begin{equation} \label{eq:main}
z_g = \!\!\!\! \sum_{k,m,r,s,a,d} 
\frac{(-1)^{k-r}(k-1)!B_r}{r!}
m^{r-1}\cdot \!\!\!\prod_{p|(m,d_1,\ldots,d_s)} \left( 1-\frac{1}{p^r}\right) 
\frac{1}{P_m^k}\prod_{i=1}^s \frac{\mu(m/d_i)^{a_i} P_{d_i}^{a_i}}{a_i!},
\end{equation}
where the sum is over integers $k,m>0$ and $r,s\ge 0$, and $s$-tuples of positive integers $a = (a_1,\ldots,a_s)$ and $d = (d_1,\ldots,d_s)$, such that
\begin{eqnarray}
0<d_1<\cdots<d_s<m, \mbox{ \ \ and \ \ } d_i \,|\, m; \label{eq:di-divides}\\
a_1+\cdots+a_s + r = k+1; \label{eq:chi-downstairs}\\
a_1d_1+\cdots+a_sd_s+g-1 = km,\label{eq:chi-upstairs}
\end{eqnarray}
and the product runs over primes $p$ dividing $m$ and all $d_1, \dots, d_s$.
\end{theorem}

\noindent The origins of the formula \eqref{eq:main} are discussed in Remark~\ref{rem:history}.
The sum is finite; see Remark~\ref{rem:finite}.  The ordinary (numerical) top weight Euler characteristics of $M_{g,n}$, which were also not previously known, may be deduced from Theorem~\ref{thm:faber-conj}; see Corollary~\ref{cor:numerical}.

\medskip

The starting point for our proof of Theorem~\ref{thm:faber-conj} is a chain complex from \cite{cgp-markedpoints-published}, denoted $K^{(g,n)}$, that computes the top weight cohomology $\Gr_{2d}^W H^*(\cM_{g,n};\QQ)$.  This chain complex is a \emph{graph complex}, whose generators are certain genus $g$ graphs with $n$ distinct marked points and some orientation data; see Section \ref{subsec:Graphs and graph complexes} below.  The $n=0$ case is treated in the previous paper \cite{cgp-graph-homology}, 
where it was used to relate the top weight cohomology of $\cM_g$ to the Grothendieck-Teichm\"uller
Lie algebra.  

The $S_n$-action on the graph complex $K^{(g,n)}$ preserves the subspaces spanned by graphs with a fixed underlying stable graph, i.e., the graph obtained by forgetting all marked points and stabilizing.   
In Section~\ref{sec:sum-over-automorphisms}, we show that the resulting contribution $z_G$ to $z_g$ from each of the finitely many underlying stable graphs $G$ is a signed sum, over the automorphisms of $G$, of monomials in the inhomogeneous power sums $P_\lambda$ and their inverses.  The pairs $(G,\tau)$, where $\tau \in \Aut(G)$, that occur in these sums may be regrouped according to the quotients $G/\tau$, understood as ``orbigraphs" in a sense made precise in Section~\ref{sec:orbigraphs}.  
The contributions of the orbigraphs are analyzed using a series of push-forwards of ``orbi-sums'' over groupoids (Lemma~\ref{lem:orbi-counting}), from which we deduce:
\begin{proposition} \label{prop:orbisum-intro}
$$z_g = \sum_{k,m,r,s,a,d} \int_{\mathrm{OG}^{\mathrm{stat},\mathrm{red}}_{(g,m,r,s,a,d)}} {\mathcal R}_* \sss \cdot \ttt \cdot \uuu,$$
where $k,m,r,s,a,d$ satisfy the conditions in Theorem~\ref{thm:faber-conj}.
\end{proposition} 
\noindent The integral denotes an orbi-sum over the groupoid $\mathrm{OG}^{\mathrm{stat},\mathrm{red}}_{(g,m,r,s,a,d)}$ defined in Lemma~\ref{lem:stat-red-coprod}, whose objects are {\em static} and {\em reduced} orbigraphs 
with associated numerical parameters $g,m,r,s,a,d$.  The terms ${\mathcal R}_* \sss$, $\ttt$, and $\uuu$ are functions on reduced static orbigraphs defined in Corollary~\ref{cor:s-t-u} and \ref{def:reduction-functor}.  The contributions of the three terms are calculated in Proposition~\ref{prop:u}, Lemma~\ref{lem:s}, and Lemma~\ref{lem:orbifold-euler}, whence the formula of Theorem~\ref{thm:faber-conj} follows.
\begin{remark}
Each numerical parameter in Theorem~\ref{thm:faber-conj} and Proposition~\ref{prop:orbisum-intro} has a natural graph theoretic interpretation, e.g., if the orbigraph $G/\tau$ is in $\mathrm{OG}^{\mathrm{stat},\mathrm{red}}_{(g,m,r,s,a,d)}$, then $g$ is the genus of $G$, $m$ is the order of $\tau$, and $r$ is the genus of the ordinary graph underlying $G/\tau$.  See Sections~\ref{sec:orbigraphs}--\ref{sec:static-structure} for details.  
\end{remark}

Similar arguments give the analogous generating functions $z_0$ and $z_1$ (summing over $n \geq 3$ and $n \geq 1$, respectively); these look slightly different from the higher genus cases.  
  \begin{proposition}\label{prop:genus0}
    $$z_0 = -P_1 \sum_{d\ge 1} \frac{\mu(d)}{d} \log P_d  +\frac{1}{2}(P_1^2-P_2).$$
  \end{proposition}
  \begin{proposition}\label{prop:genus1}
    $$z_1 = -\frac{1}{2} \sum_{d\ge 1} \frac{\phi(d)}{d} \log P_d-\frac{1}{4}\frac{P_1^2}{P_2}+P_1-\frac{3}{4}.$$
  \end{proposition}
  \noindent Here, $\phi$ denotes the Euler totient function. The second author first derived these formulas from Getzler's work \cite{getzler-semi-classical}.  In Section~\ref{sec:genus-0-1} we rederive them by combining our method with the work of Robinson--Whitehouse on tree spaces \cite[Theorem 3.1]{robinson-whitehouse-tree} and \cite[Theorem 1.2]{cgp-markedpoints-published}, respectively.

\medskip

\begin{remark} \label{rem:finite}
 The sum in \eqref{eq:main} is finite, because conditions \eqref{eq:di-divides}--\eqref{eq:chi-upstairs} imply bounds on $k$ and $m$ in terms of $g$, namely $k \leq g $ and $m \leq 2g + 2$.  To see this, first note that if $s = 0$, then $km = g-1$.  Assume $s > 0$.  Then $m \geq 2$ and $d_s \leq m/2$.  Combining \eqref{eq:chi-downstairs} and \eqref{eq:chi-upstairs} gives
\[
km \leq d_s \sum a_i + g - 1 \leq (k+1)m/2 + g - 1,
\]
and hence $(k-1)m \leq 2g - 2$.  This proves the claim for $k \neq 1$.   If $k = 1$, then $r \in \{0,1\}$ and $a_i, s \in \{ 1, 2 \}$.  These cases may be checked by hand.

In particular, we see that $z_g$ lies in the subring $\QQ[P_1^{\pm 1}, \dots, P_{2g+2}^{\pm 1}] \subset \widehat \Lambda$ when $g \geq 2$.
\end{remark}

\begin{remark}
Our approach to proving this formula was inspired in part by Gorsky's computation of the $S_n$-equivariant Euler characteristic of the full cohomology of $\cM_{g,n}$ \cite{gorsky-equivariant}, which builds on the work of Harer and Zagier \cite{harer-zagier-euler}.  Briefly, our graphs are analogous to the surfaces in \cite{gorsky-equivariant} and \cite{harer-zagier-euler}, using \cite{cgp-markedpoints-published} to relate graphs to top weight cohomology; in this analogy, our Section~\ref{sec:sum-over-automorphisms} corresponds roughly to the main technical result \cite[Theorem 3.3]{gorsky-equivariant} while our Sections \ref{sec:orbigraphs}, \ref{sec:static}, and \ref{sec:conclusion} correspond roughly to \cite[Theorems 4 and 5]{harer-zagier-euler}.  We now discuss more closely a comparison to Gorsky's methods; see also Remarks~\ref{rem:laurent-monomials} and~\ref{rem:connection-to-gorsky-superficial} for further discussion.

Gorsky considers the forgetful map of varieties $\cM_{g,n} \to \cM_g$, whose fiber over $C \in \mathcal{M}_g$ is $F(C,n)/\!\Aut(C)$; here $F(C,n)$ denotes the configuration space.  The $S_n$-equivariant Euler characteristic of the fiber, or more precisely its generating function, can be written as a sum of contributions for each $\tau \in \Aut(C)$.  The contribution to this sum of a pair $(C,\tau)$ is a monomial in the symmetric functions $P_i$; exactly which monomial it is turns out to depend only on the Euler characteristics of the spaces, denoted $C_j(\tau)$, parametrizing those points in $C$ whose $\tau$-orbit has size exactly $j$.  In this way, using the technique of integration against Euler characteristic (see \cite{viro-some}) for purposes of rearrangement, the generating function can be expressed as a sum of contributions from the auxiliary moduli spaces
\[\mathcal{M}_{g}(k_1,\ldots,k_r) = \{(C,\tau): \chi(C_j(\tau)) = jk_j \text{ for all }j\}.\]
The contribution from $\mathcal{M}_{g}(k_1,\ldots,k_r)$ is exactly $\chi^{\mathrm{orb}}(\mathcal{M}_{g}(k_1,\ldots,k_r)) \cdot \prod_j P_j^{k_j}.$ It is then possible to calculate these orbifold Euler characteristics directly from the orbifold Euler characteristics of spaces $\mathcal{M}_{g',n'}$ for appropriate pairs $(g',n')$.   

Note that Gorsky's argument cannot be specialized thoughtlessly to top-weight cohomology. For example, given an algebraic fiber bundle $F\to E\to B$, the Euler characteristic is multiplicative---this simple statement is one of the keys to the technique of integration against Euler characteristic---whereas the top-weight Euler characteristic is not. The argument presented in this paper restores the spirit of Gorsky's geometric arguments on the level of tropical geometry, with spaces of graphs replacing spaces of curves.  

Actually, the main arguments here are presented in a bare-bones way, in terms of finite categories: as orbisums over groupoids of graphs, calculated by pushing forward along appropriate functors.   While we did not choose to do so here, it would be possible to rephrase the arguments in this paper more geometrically, to parallel Gorsky's argument more closely.  After removing the loci of tropical curves with positive vertex weights and/or repeated markings, one obtains a forgetful map of tropical moduli spaces whose fiber over a genus $g$ graph $G$ is $F(G,n)/\!\Aut(G)$.  (This perspective was employed recently in \cite{bibby-chan-gadish-yun-homology} to obtain new computations of $\Gr^W_{2d}H^*(\mathcal{M}_{2,n};\QQ)$.)  Then arguments similar to the ones reviewed above can be used.  Significant additional complications, not present in the situation of \cite{gorsky-equivariant}, necessarily arise due to the more complicated nature of the tropical analogues of the spaces $C_j(\tau)$ mentioned above.  In brief: when $C$ is an algebraic curve, these spaces are either finite sets of points, or the complement thereof in $C$. When $C$ is a graph, more complicated loci can arise.  These complications are handled in the latter part of this paper.
\end{remark}

\begin{remark}
A closely related formula was obtained by Tsopm\'en\'e and Turchin in their work on rational homotopy groups of spaces of long links \cite[Theorem 1.3]{MR3806571}.  For $g \geq 2$ and $n \geq 1$, the genus $g$ summand of the graph complex $M(P_2^n)$ that appears in their work is quasi-isomorphic to the complex $K^{(g,n)}$ considered here. To see this, note that $M(P_2^n)$ involves graphs that may have loops and repeated markings, but not vertices of positive weight, and is thus isomorphic to the cellular chain complex of $\Delta_{g,n}$ relative to the subcomplex $\Delta_{g,n}^w$ of graphs with some vertex of positive weight.  The subcomplex $\Delta_{g,n}^w$ is contractible \cite[Theorem~1.1(1)]{cgp-markedpoints-published}, and this yields the quasi-isomorphism. (Other minor differences that do affect the Euler characteristic occur for $g \leq 1$; most importantly, $\Delta_{g,n}^w$ is empty when $g = 0$ or $(g,n) = (1,1)$.)  

It is remarkable that such closely related graph complexes appear in such distant contexts. Not only are the contexts and motivations for these two works apparently unrelated, but also the methods of proof have little in common. The work of Tsopm\'en\'e and Turchin is more algebraic. The category of functors from finite sets and surjections to $\QQ$-vector spaces plays a central role, and the arguments make significant use of the formalism of modular operads from \cite{getzler-kapranov}, as compared with the geometric approach featuring moduli spaces of graphs used here. 

Moreover, the formulas have dramatically different features. Each has advantages over the other.  The formula of Tsopm\'en\'e and Turchin combines all $g$, and expresses the corresponding $S_n$-equivariant Euler characteristics simultaneously in a single expression.  On the other hand, our approach works equally well for $n = 0$ and proves that, for $g \geq 2$, the generating function $z_g$ is a finite linear combination of Laurent monomials of degree $1 - g$ in the inhomogeneous power sum symmetric functions $P_m$ (graded such that $P_m$ has degree $m$), as conjectured by the second author over a decade ago.  The problem of computing the Euler characteristic of graph complexes for $n = 0$ was also considered earlier by Willwacher and \v{Z}ivkovi\'{c}. Their $\tilde{\chi}^{even*}_{g-1}$ agrees with the Euler characteristic of $K^{(g,0)}$.  They computed these quantities for $g \leq 31$ and remarked on the problem of finding a ``pretty" formula \cite{willwacher-zivkovic}.

We are most grateful to T.~Willwacher for bringing \cite{MR3806571} to our attention and to P.~Tsopm\'en\'e and V.~Turchin for helpful discussions regarding the relations between the two formulas and the graph complexes that give rise to them. We hope that future work will further illuminate the relations between the formulas and the reasons why these quasi-isomorphic graph complexes appear in such  distant corners of the mathematical world.
\end{remark}

\begin{remark} \label{rem:history}
The formula of Theorem \ref{thm:faber-conj} originates in a sense from
the intriguing remark ``Unfortunately, our formula for 
$\mathsf{e}^{S_n}(\overline{\cM}_{1,n})$ does not render [Poincar\'e] duality
manifest'' of Getzler \cite[p.~489]{getzler-semi-classical} and from his
use of Poincar\'e duality in the proof of Prop.~16 in
\cite{getzler-trr}. In several further computations, Poincar\'e duality
provided a nontrivial check, which led to the question of what it implies
for the cohomology of $\cM_{g,n}$. The formalism of modular operads
\cite[Thm.~8.13]{getzler-kapranov} for computing the $S_n$-equivariant
Euler characteristics of moduli spaces of stable pointed curves
from those of moduli spaces of smooth pointed curves (or vice versa)
is compatible with weights. Since the top weight Euler characteristic
of $\overline{\cM}_{g,n}$ is trivial, the top weight Euler characteristic
of $\cM_{g,n}$ can in principle be determined. A further study of the
underlying geometry by the second author led to the (computer-aided)
computation of $z_g$ for $2\leq g\leq 8$. At a conference in 2008, he shared
these results with Zagier, who then conjectured the formula of 
Theorem~\ref{thm:faber-conj}. Although the case $g=9$ could be verified
a little later,
no further progress was made since then. 
The proof of Theorem~\ref{thm:faber-conj} given
here was obtained by the first, third, and fourth authors. 

We also remark that \cite[Theorem~8.13 and Corollary~8.15]{getzler-kapranov} do in fact yield a formula for the top-weight Euler characteristic of $\cM_{g,n}$, in the sense that all of these top-weight Euler characteristics can in principle be extracted computationally from these theorems.  In this paper, we provide in Theorem~\ref{thm:faber-conj} a closed formula for each generating function $z_g$, which may be computed genus by genus, and more efficiently than if proceeding directly from \cite{getzler-kapranov}.  Another feature of Theorem~\ref{thm:faber-conj} worth mentioning, and which is not transparent from \cite{getzler-kapranov}, is the remarkable structure,  discovered in the second author's original computations, that $z_g$ is a finite linear combination of Laurent monomials of degree $1-g$ in the inhomogeneous power sum symmetric functions $P_i$, with all denominators being of the form $P_m^k$. This structure is closely parallel to the structure of the corresponding generating function for the full $S_n$-equivariant Euler characteristic of $\mathcal{M}_{g,n}$, as computed by Gorsky \cite{gorsky-equivariant}. Indeed, it reflects the geometry of the maps $M_{g,n}^{\mathrm{trop}} \to M_{g}^{\mathrm{trop}}$, just as the parallel structure of Gorsky's computations reflects the geometry of the algebraic maps $\mathcal{M}_{g,n} \to \mathcal{M}_g$. 

One may fruitfully study other graded pieces of the weight filtration on the cohomology of $\mathcal{M}_{g,n}$.  See, for instance, \cite{pw} for results on the weight 2 compactly supported cohomology of these moduli spaces.  
There are still natural formulas for the $S_n$-equivariant Euler characteristic of weight 2 compactly supported cohomology of $\mathcal{M}_{g,n}$, and these can be expressed in terms of the inhomogeneous power sums $P_i$, but the expressions are no longer finite, and no longer homogeneous.  See \cite[Theorem~1.1 and Corollary~1.2]{pw-chi}. 

\end{remark}

\noindent \textbf{Acknowledgments.} 
We are very grateful to D.~Zagier for sharing with us his conjectural
formula \eqref{eq:main} for $z_g$.  We also thank the referee for a number of insightful comments and suggestions.

MC is grateful to D.~Cartwright for helpful preliminary discussions.
MC was supported by NSF DMS-1701924, NSF CAREER DMS-1844768, the Henry Merritt Wriston Fellowship, and a Sloan Research Fellowship.  
CF is grateful to K.~Consani and E.~Getzler for helpful preliminary discussions.
CF was supported by NWO EW-613.001.651.
SG was supported by the European Research Council (ERC) under the European Union's Horizon 2020 research and innovation programme (grant agreement No 682922) and by the Danish National Research Foundation (DNRF92 and DNRF151). 
SP was supported by NSF grants DMS-2001502 and DMS-2053261 and a Simons Fellowship.  He thanks UC Berkeley and MSRI for their hospitality and ideal working conditions.


\section{Preliminaries}

All of the representations, symmetric functions, and graph complexes that we consider are with rational coefficients.  We begin by recalling a few basic facts that will be used throughout.  

\subsection{Equivariant Euler characteristics}

For any partition $\lambda \vdash n$, let $V_\lambda$ denote the corresponding irreducible representation of $S_n$, and $s_\lambda$ the corresponding Schur function, which is an element of the ring of symmetric functions.  Then $\coprod_{n\ge 0} \{s_\lambda \colon \lambda\vdash n\}$ is a vector space basis for $\Lambda$, where $$\Lambda = \lim_{\longleftarrow} \QQ[x_1,\ldots,x_n]^{S_n}$$ is the ring of symmetric functions (the inverse limit is calculated in the category of graded rings).  Another natural basis for the ring of symmetric functions is given by the power sum symmetric functions: for each $i>0$, let $p_i = \sum_{j>0} x^i_j$; and for any $n\ge 0$ and partition $\lambda = (\lambda_1\ge \cdots \ge \lambda_r) \vdash n$, let $p_\lambda = p_{\lambda_1}\cdots p_{\lambda_r}$.
Then $\coprod_{n\ge 0} \{p_\lambda \colon \lambda\vdash n\}$ is also a basis for $\Lambda$.

It will also be convenient to define the inhomogeneous power sum symmetric functions as
$$P_i = 1+p_i, \qquad P_\lambda = P_{\lambda_1}\!\cdots P_{\lambda_r} \text{ for } \lambda = (\lambda_1\ge \cdots\ge \lambda_r) \text{ a partition.}$$

Let $V = \{V^n\}_{n\ge 0}$ be a sequence of graded virtual $S_n$-representations, where $V^n$ has graded pieces $V^n = \bigoplus_{i\ge 0} V^n_i$, and $$V^n_i = \bigoplus_{\lambda\,\vdash\,n} c^n_{i,\lambda} V_\lambda\qquad\text{for}\quad c^n_{i,\lambda} \in \ZZ_{\geq 0}.$$  
\defnow{The Frobenius characteristic of $V$ is the symmetric function
$$z_V = \sum_{i,n\ge 0} (-1)^i c^n_{i,\lambda} s_\lambda \quad\in \widehat \Lambda = \varprojlim_n \QQ\llbracket x_1, \dots, x_n \rrbracket^{S_n}.
$$
}
A useful expression for $z_V$ in terms of power sum symmetric functions is as follows. Given $\sigma \in S_n$ with cycle type $a_1\ge a_2\ge \cdots \ge a_\ell$, let 
\begin{equation} \label{eq:psi-notation}
\psi(\sigma) = p_{a_1}\cdots p_{a_\ell}.
\end{equation} 
Then it is well known that

$$z_V = \sum_{i,n\ge 0} \sum_{\sigma \in S_n} \frac{(-1)^i}{n!}  \chi_{V^n_i}(\sigma) \psi(\sigma),$$
where $\chi_{V^n_i}(\sigma)$ denotes the character of $V^n_i$ on the conjugacy class $\sigma$ of $S_n$.

\medskip

\subsection{Graphs and graph complexes}\label{subsec:Graphs and graph complexes}

To us a \emph{graph} $G$ is given by two finite sets $V(G)$ and $H(G)$, an involution $s:H(G) \to H(G)$ without fixed points, and a map $r: H(G) \to V(G)$.  An \emph{isomorphism} $\tau: G \to G'$ is given by bijective set maps $\tau_V: V(G) \to V(G')$ and $\tau_H: H(G) \to H(G')$ such that $\tau_H \circ s = s' \circ \tau_H$ and $\tau_V \circ r = r' \circ \tau_H$.  In \cite{cgp-markedpoints-published} we also considered some non-invertible morphisms between graphs, but in this paper we shall only need the isomorphisms.  The elements of $V(G)$ are called vertices, the elements of $H(G)$ are called half-edges, and we let $E(G) = H(G)/(x \sim s(x))$ whose elements we call edges.
The \emph{geometric realization} of $G$ is the topological space $|G| = (V(G) \amalg ([-1,1] \times H(G)) / \!\sim$, where the equivalence relation is generated by $(t,x) \sim (-t,s(x))$ and $r(x) \sim (1,x)$ for $x \in H(G)$ and $t \in [-1,1]$.  See op.\ cit.\ for more details and motivation.

An $n$-marked graph is a pair $(G,m)$ where $G$ is a graph and $m$ is a function $\{1, \dots, n\} \to V(G)$ which we call the \emph{marking}.  It is \emph{stable} if $|r^{-1}(v) \amalg m^{-1}(v)| \geq 3$ for all $v \in V(G)$, and \emph{connected and genus $g$} if $|G|$ is connected and has Euler characteristic $1-g$.

\defnow{
  Let $\J_{g,n}^\circ$ be the groupoid whose objects are stable, connected, genus $g$, $n$-marked graphs $(G,m)$ as above, satisfying additionally that $m: \{1, \dots, n\} \to V(G)$ is injective.

  Isomorphisms $\tau: (G,m) \to (G',m')$ in $\J_{g,n}^\circ$ are graph isomorphisms $\tau: G \to G'$ satisfying $\tau_V \circ m = m'$.
}

This groupoid is essentially small, and we shall tacitly replace it
with an equivalent small subcategory.  In the category $\J_{g,n}$
defined in \cite[Section 2.2]{cgp-markedpoints-published} we did not require $m$ to be injective, and we also included the data of a \emph{weighting} $w: V(G) \to \NN$.  We may identify $\J_{g,n}^\circ$ with a subcategory of $\J_{g,n}$ by setting $w = 0$.  We also set the notation $\J_g^\circ = \J_{g,0}^\circ.$

\defnow{
  For a graph $G$ and an automorphism $\tau: G \to G$, we write $\sgn(\tau_E) \in \ZZ^\times$ for the sign of the induced permutation $\tau_E$ on $E(G)$.

  A graph $G$ has \emph{alternating automorphisms} if $\sgn (\tau_E) = 1$ for all automorphisms $\tau$.
}

Finally, let us recall the chain complex $\Kgn $ associated to the category $\J_{g,n}^\circ$.  It has generators $[G,m,\omega]$ where $(G,m) \in \J_{g,n}^\circ$ is an object and $\omega$ is a total order on $E(G)$, of homological degree the cardinality of $E(G)$.  These generators are subject to the relation that $[G,m,\omega] = \pm [G',m',\omega']$ if there exists an isomorphism $\tau: (G,m) \to (G',m')$ in $\J_{g,n}^\circ$; the sign ``$\pm$'' is the sign of the permutation determined by the bijection $\tau_E: E(G) \to E(G')$ and the total orders $\omega$ and $\omega'$.  In particular we obtain a basis of $\Kgn $ by choosing one object $(G,m)$ in each isomorphism class with alternating automorphisms, and for each such choosing a total order of $E(G)$.  The boundary homomorphism  on $\Kgn$ is given by signed sum of one-edge contractions and has degree $-1$. Its precise definition, which we do not need here, is in \cite[\S5.2]{cgp-markedpoints-published}.

The following proposition is proved in \cite[Proposition 5.3]{cgp-markedpoints-published}.  The proof given there provides an explicit isomorphism.  Briefly, we identified $\Delta_{g,n}$, the moduli space of tropical curves of volume 1, with the dual complex of the boundary divisor in the stable curves compactification $\overline{\cM}_{g,n}$.  Then we proved that the cellular chain complex on $\Delta_{g,n}$, with degrees shifted by one, is quasi-isomorphic to $\Kgn$.

\propnow{\label{prop:kgn}
  For $g > 0$ and $n \geq 0$ with $2g-2+n > 0 $ and $(g,n) \neq (1,1)$, with the boundary homomorphism $\partial: \Kgn  \to \Kgn $ defined in \cite{cgp-markedpoints-published}, there is an $S_n$-equivariant isomorphism
  \begin{equation*}
    H_*(\Kgn ,\partial) \cong \Gr_{2d}^W H^{2d-*}(\cM_{g,n};\QQ).
  \end{equation*}
}

\cornow{
 The $S_n$-equivariant top weight Euler characteristic of $\cM_{g,n}$ is given by
  \begin{equation*}
    z_g = \sum_{i,n\ge0} \sum_{\sigma \in S_n} \frac{(-1)^i}{n!} \chi_{\Kgn_i}(\sigma) \psi(\sigma)
  \end{equation*}
}
\begin{proof}
  The Frobenius characteristic of the homology of $\Kgn$ is equal to the Frobenius characteristic on the chain level, as for any Euler characteristic.
\end{proof}

\begin{remark}
Note that $z_g$ is also the generating function for the $S_n$-equivariant weight zero compactly supported Euler characteristics of $\cM_{g,n}$, since Frobenius characteristics are invariant under duality.
\end{remark}

\noindent For later use we define a larger groupoid, whose morphisms are allowed to permute the markings.
\defnow{
  Let $\J_{g,n}^\circ/S_n$ be the groupoid with the same objects as $\J_{g,n}^\circ$, but whose isomorphisms $\tau: (G,m) \to (G',m')$ are the isomorphisms of graphs $\tau: G \to G'$ such that $\tau_V \circ m = m' \circ \sigma$ for some permutation $\sigma \in S_n$.  
  We shall write $\tau_{\vert m} \in S_n$ for the (uniquely determined) permutation satisfying $\tau_V \circ m = m \circ \tau_{\vert m}$.
}
\noindent The reader may recognize $\J_{g,n}^\circ/S_n$ as an instance of the Grothendieck construction, applied to the functor $S_n \to \mathsf{Cat}$ sending the unique object of $S_n$ to $\J_{g,n}^\circ.$ 

\subsection{Orbi-counting and orbi-summation}
\label{sec:orbi-counting-adding}

Let us make some elementary observations about counting objects with automorphisms.  We shall use these as organizing principles for later arguments in this paper.  By a finite groupoid we mean a groupoid that is equivalent to one with a finite number of objects and morphisms. 

Let $\pi_0(\mathcal{G})$ denote the set of isomorphism classes in a finite groupoid $\mathcal{G}$, let $V$ be a rational vector space, and let $f: \pi_0(\mathcal{G}) \to V$ be a function.  We denote the orbisum of $f$ by
\begin{equation*}
  \int_\mathcal{G} f = \int_{x \in \mathcal{G}} f(x) := \sum_{[x] \in \pi_0(\mathcal{G})} \frac{f(x)}{|\Aut(x)|} \in V.
\end{equation*}
In particular,  the rational number $\int_{\mathcal{G}} 1$ is the \emph{groupoid cardinality} of $\mathcal{G}$. 

\begin{remark} \label{rem:inf}
 Suppose $\pi_0(\mathcal{G})$ is infinite.  If $V$ is complete with respect to a filtration $V \supset F^0V \supset F^{1}V \supset \dots$ and $\pi_0(\mathcal{G}) \setminus f^{-1}(F^i V)$ is finite for all $i$, then the definition of $\int_{\mathcal{G}} f$ still makes sense, and the lemmas and discussion below extend essentially verbatim.  We will use orbisummation in this level of generality as needed, without further mention.
\end{remark}

If $F: \mathcal{G} \to \mathcal{H}$ is a functor between finite groupoids and $f: \pi_0(\mathcal{G}) \to V$ is a function, we define the push-forward $(F_* f): \pi_0(\mathcal{H}) \to V$ by the formula
\begin{equation}\label{eq:count-over-comma}
  (F_*f)([h]) = \int_{(F\downarrow h)} f,
\end{equation}
where the subscript denotes the ``comma category'' $(F \downarrow h)$. Recall that the objects of $(F\downarrow h)$ are pairs $(g,\phi)$ with $g$ an object of $\mathcal{G}$ and $\phi: F(g) \to h$ a morphism in $\mathcal{H}$, and morphisms $(g,\phi) \to (g',\phi')$ are morphisms $j\colon g \to g'$ in $\mathcal{G}$ such that  $\phi' \circ F(j) = \phi$. Regard $f$ as a function on $\pi_0(F\downarrow h)$ by composing with the natural map $\pi_0(F\downarrow h) \to \pi_0(\mathcal{G})$.  

\begin{lemma}
The push-forward $F_*f$ is equivalently characterized by
\begin{equation}\label{eq:more-explicitly}
(F_*f)([h]) = |\!\Aut_{\mathcal{H}}(h)| \sum \frac{f(g_i)}{|\!\Aut_\mathcal{G}(g_i)|}, 
\end{equation}
where the sum is over objects $g_i \in \mathcal{G}$, one in each isomorphism class of $\mathcal{G}$ with $F(g_i) \in [h] \in \pi_0(\mathcal{H})$.  
\end{lemma}

\begin{proof}
Reduce to the case that $\mathcal{H}$ has one object $h$ and $\mathcal{G}$ is skeletal, and let $g\in \mathcal{G}$.  Under the $\Aut_\mathcal{G}(g)$-action on $\Aut_{\mathcal{H}}(h)$ given by $\phi \cdot \psi = \phi\circ F(\psi)$, the stabilizer of $\phi\in\Aut_{\mathcal{H}}(h)$ is isomorphic to $\Aut_{(F\downarrow h)}(g,\phi)$, and the orbit of $\phi$ bijects with the isomorphism class $[(g,\phi)]$.  So
$$F_*(f)(h) = \sum_{g\in\mathcal{G}}\sum_{\phi\in\Aut_\mathcal{H} (h)} \frac{f(g)}{|[(g,\phi)]||\!\Aut_{(F\downarrow h)}(g,\phi)|}=|\!\Aut_\mathcal{H}(h)|\sum_{g\in\mathcal{G}} \frac{f(g)}{|\!\Aut_{\mathcal{G}}(g)|},$$
as required.
\end{proof}

The following lemma is a straightforward consequence of~\eqref{eq:more-explicitly}.
\lemnow{\label{lem:orbi-counting}
  Let $F: \mathcal{G} \to \mathcal{G}'$ and $F': \mathcal{G}' \to \mathcal{G}''$ be functors between finite groupoids and let $f: \pi_0(\mathcal{G}) \to V$ be a function to a rational vector space.  Then
  \begin{equation*}
    F'_* (F_* f) = (F' \circ F)_* f.
  \end{equation*}
  In particular, letting $\mathcal{G}''$ be trivial, $$\int_{\mathcal{G}} f = \int_{\mathcal{G}'} (F_* f).$$ 
  }

If $F: \mathcal{G} \to \mathcal{H}$ is a functor between finite groupoids and $\alpha: \pi_0(\mathcal{H}) \to \QQ$ is a function, we similarly define $F^*\alpha = \alpha \circ (\pi_0(F))\colon \pi_0(\mathcal{G})\to\QQ$.  The sets of functions $\pi_0(\mathcal{H}) \to \QQ$ and $\pi_0(\mathcal{G}) \to \QQ$ are rings under pointwise product, and the sets of functions $\pi_0(\mathcal{H}) \to V$ and $\pi_0(\mathcal{G}) \to V$ are modules over those two rings, respectively.  Then $F^*$ is a ring homomorphism, and it is easily verified that $F_*$ is a homomorphism of modules over $\mathrm{Map}(\pi_0(\mathcal{H}),\QQ)$, i.e., that the equation
\begin{equation}\label{eq:13}
  F_*((F^*\alpha) \cdot f)= \alpha \cdot(F_*f) 
\end{equation}
holds for each $f: \pi_0(\mathcal{G}) \to V$.  In particular, we have
\begin{equation}\label{eq:4}
  F_*F^* \alpha = \alpha \cdot (F_*F^* 1).
\end{equation}
The function $F_*F^* 1_H = F_* 1_G: \pi_0(\mathcal{H}) \to \QQ$ measures the groupoid cardinalities of the comma categories $(F\downarrow h)$. 
\exnow{\label{example:free-loops} For a finite groupoid $\mathcal{G}$, let
  $\mathrm{Fun}(\ZZ,\mathcal{G})$ be the groupoid whose objects are functors from the group $\ZZ$,
  regarded as a one-object groupoid, to $\mathcal{G}$ and whose morphisms are the natural isomorphisms.
  Objects shall be written $(x,\tau)$, where $x$ is the image of the one object, and  $\tau \in \Aut_\mathcal{G}(x)$ is the image of the generating morphism $1 \in \ZZ$.
  There is a forgetful functor
  \begin{equation*}
    U: \mathrm{Fun}(\ZZ,\mathcal{G}) \to \mathcal{G}
  \end{equation*}
  given on objects by $(x,\tau) \mapsto x$.  A function $f:\pi_0(\mathrm{Fun}(\ZZ,\mathcal{G})) \to V$ associates to each $(x,\tau)$ an element $f(x,\tau) \in V$, which depends only on the conjugacy class of $\tau \in \Aut(x)$ and the isomorphism class of $x$.  Then from~\eqref{eq:count-over-comma}, we see that $U_*f: \pi_0(\mathcal{G}) \to V$ is given by
  \begin{equation*}
    (U_* f)([x]) = \sum_{\tau \in \Aut_\mathcal{G}(x)} f(x,\tau),
  \end{equation*}
  because the comma category $(U \downarrow x)$ is equivalent to the set $\Aut_\mathcal{G}(x)$, regarded as a discrete category, i.e., a category in which all morphisms are identities.  The orbisum of $f$ then becomes
  \begin{equation*}
    \int_{\mathrm{Fun}(\ZZ,\mathcal{G})} f = \int_{\mathcal{G}} U_*f = \sum_{[x] \in \pi_0(\mathcal{G})} \frac1{|\Aut_\mathcal{G}(x)|} \sum_{\tau} f(x,\tau),
  \end{equation*}
  where the first sum ranges over objects $x \in \mathcal{G}$, one in each isomorphism class, and the second over all automorphisms of $x$.
}


\exnow{\label{ex:frob-char-as-orbisum}
  Let $V$ be a sequence of graded $S_n$ representations, regarded as functor from (a skeletal subcategory of) finite sets and bijections to finite-dimensional graded rational vector spaces, and let
  \begin{equation*}
    f_V: \pi_0(\mathrm{Fun}(\ZZ,\mathrm{FinSet})) \to \QQ
  \end{equation*}
  be the function that sends a finite set $S$ and a bijection $\sigma: S \to S$ to the supertrace of $\sigma$ acting on $V$ (i.e., alternating sum of the traces in each degree).  Then the Frobenius characteristic is
  \begin{equation*}
    z_V = \int_{\mathrm{Fun}(\ZZ,\mathrm{FinSet})} \psi \cdot f_V,
  \end{equation*}
  where $\psi: \pi_0(\mathrm{Fun}(\ZZ,\mathrm{FinSet})) \to \widehat\Lambda$ is the function given by~(\ref{eq:psi-notation}).
}

For later use we record the following observation, whose proof we leave as an exercise.  

\lemnow{\label{lem:fiber-free-loop-space}
  Let $F: \mathcal{G} \to \mathcal{H}$ be a functor between finite groupoids, and assume that $(F\downarrow h)$ is equivalent to a discrete category for all objects $h \in \mathcal{H}$.  Then the induced functor
  \begin{equation*}
    LF: \mathrm{Fun}(\ZZ,\mathcal{G}) \to \mathrm{Fun}(\ZZ,\mathcal{H}),
  \end{equation*}
  given by composing with $F$, has the same property: the comma category $(LF\downarrow (h,\tau))$ over any object $(h,\tau) \in \mathrm{Fun}(\ZZ,\mathcal{H})$ is equivalent to a discrete category.  Moreover, the set $\pi_0(LF\downarrow (h,\tau))$ may be identified with the $\tau$-fixed elements of the set $\pi_0 (F\downarrow h)$.
  }

\noindent We will apply this lemma in the proof of Proposition~\ref{prop:start}.

\section{The contribution of a graph as a sum over automorphisms}  \label{sec:sum-over-automorphisms}

Recall that $\Kgn _i$ has a basis with one generator for each isomorphism class of stable, connected, genus $g$, $n$-marked graphs with $i$ edges, and alternating automorphisms.  This basis is not quite canonical; the sign of each element depends on a choice of ordering of the edges, up to alternating permutation, but these signs will not affect our calculations.

\defnow{
  For a bijection $\tau: S \to S$ of a finite set $S$, write
  \begin{equation*}
    P(\tau) = P_{\lambda_1} \dots P_{\lambda_s} \in \widehat\Lambda,
  \end{equation*}
  where $\lambda_1 \geq \dots \geq \lambda_s$ is the cycle type of $\tau$ and $P_a = 1 + p_a \in \widehat\Lambda$ for $a \in \ZZ_{>0}$.
}

\begin{prop}  \label{prop:start}
  The element $z_g \in \widehat\Lambda$ may be expressed as
  \begin{equation*}
    z_g = \sum_{G \in \pi_0(\J_g^\circ)} \frac{z_G}{|\Aut(G)|},
  \end{equation*}
  where the sum is over one $G$ in each isomorphism class, and
  \[
    z_G =  (-1)^{|E(G)|} \sum_{\tau \in \Aut(G)} {\rm sgn}(\tau_E) \frac{P(\tau_V) P(\tau_E)}{P(\tau_H)},
  \]
where $\tau_V$, $\tau_E$, and $\tau_H$ denote the permutations $\tau$ induces on the finite sets $V(G)$, $E(G)$, and $H(G)$, respectively.
\end{prop}

Note that, by Example~\ref{example:free-loops}, the two sums $\sum_{G} \frac1{|\Aut(G)|} \sum_{\tau \in \Aut(G)}$ may be combined into one orbisum over the groupoid $\mathrm{Fun}(\ZZ,\J_g)$.  We shall prove the proposition by finding a formula for $z_g$ as an orbisum over a different groupoid, and then applying Lemma~\ref{lem:orbi-counting}.  
We shall employ this language of orbisums over groupoids here as a warmup for later arguments.

\begin{remark}\label{rem:laurent-monomials}
Before we proceed to the proof of the proposition, let us comment on the relationship with Gorsky's computations in \cite{gorsky-equivariant}. Gorsky likewise shows that, for fixed $g \geq 2$, the generating function for the $S_n$-equivariant Euler characteristic of $\mathcal{M}_{g,n}$ is a finite linear combination of Laurent monomials in the power sum symmetric functions $P_i$. Moreover, setting the degree of $P_i$  to be $i$, each of these monomials has the following properties:
\begin{enumerate}
\item the denominator is a pure power $P_m^k$,
\item each $P_\ell$ appearing in the numerator has $\ell | m$,
\item the total degree is $2-2g$.
\end{enumerate}
Indeed, in Gorsky's computations, the space of pairs $(C,\tau)$, where $C$ is an algebraic curve of genus $g$ and $\tau \in \Aut(C)$,
contributes to the Laurent monomial
\[\prod_j P_j^{\chi_c (C_j(\tau))/j},\]
where $C_j(\tau)$ denotes the locus of points in $C$ of $\tau$-orbit size exactly $j$.   The contribution to the coefficient of a given monomial is then the orbifold Euler characteristic of the moduli space of such pairs $(C,\tau)$. 


Similarly, we show that each 
pair  $(G, \tau)$, where $G$ is the dual graph of a stable curve of genus $g$ and $\tau \in \Aut(G)$, contributes to exactly one monomial, namely
$\prod_j P_j^{k_j}$ where $k_j$ is the compactly supported Euler characteristic of the image in $G/\tau$ of the locus of points in $G$ whose $\tau$-orbit has size  exactly $j$.  The contribution to the corresponding coefficient is, up to sign, $\frac{1}{|\Aut(G)|}$, and summing over all pairs $(G,\tau)$ with the same parameters $k_j$, the total contribution is an orbifold Euler characteristic of the moduli space of such pairs.  In general, controlling which Laurent monomials appear can be difficult, because, unlike for curves, the locus of points in a graph $G$ whose $\tau$-orbit has size some fixed $j < \ord(\tau)$ can be more complicated than a finite set.  However, we shall see that all contributions from pairs $(G,\tau)$ cancel except for those for which $G/\tau$ is a {\em reduced, static orbigraph}; see Section~\ref{sec:orbigraphs} and thereafter. For these,  the monomials appearing in $z_g$ do satisfy properties $(1)$ and $(2)$, and an Euler characteristic computation shows that each is homogeneous of degree $1-g$.
\end{remark}

\begin{proof}[Proof of Proposition~\ref{prop:start}]
  In this proof, notation like $\sum_{x \in \pi_0(\mathcal{G})} f(x)$ shall always mean the sum over one object $x \in \mathcal{G}$ in each isomorphism class in the groupoid $\mathcal{G}$.  

  Let $\J_{g,n}^+ \subset \J_{g,n}^\circ$ be the full sub-groupoid whose objects are those with alternating automorphisms.  A permutation $\sigma \in S_n$ defines a functor $\J_{g,n}^+ \to \J_{g,n}^+$ which we shall denote $G \mapsto G.\sigma$, and this functor in turn permutes the set $\pi_0(\J_{g,n}^+)$ of isomorphism classes.  The chain complex $\Kgn $ has one basis element for each isomorphism class in $\J_{g,n}^+$, canonically determined up to a sign, and with respect to this basis the action of $\sigma$ is a signed permutation.  Non-zero diagonal entries appear for $G \in \J_{g,n}^+$ when there exists an isomorphism $\phi: G \cong G.\sigma$, and in this case the entry is $\sgn(\phi_E) \in \ZZ^\times$.  The condition that $G$ has alternating automorphisms ensures that this sign is well defined, independent of the choice of $\phi$.  Then the supertrace of $\sigma: \Kgn \to \Kgn$ is
  \begin{equation*}
    \sum_{\substack{G \in \pi_0(\J^{+}_{g,n}) \\ \exists \phi\,\colon \!G\xrightarrow{\cong}G. \sigma}}(-1)^{|E(G)|} \sgn(\phi_E).
  \end{equation*}
  so we get
  \begin{equation*}
    z_g = \sum_n \sum_{\sigma \in S_n} \sum_{\substack{G \in \pi_0(\J^{+}_{g,n}) \\ \exists \phi\,\colon \!G\xrightarrow{\cong}G. \sigma}} \frac1{n!} (-1)^{|E(G)|} \sgn(\phi_E) \psi(\sigma) \in \widehat\Lambda.
  \end{equation*}
  When an isomorphism $\phi: G \cong G.\sigma$ exists, then there exist precisely $|\Aut_{\J_{g,n}^\circ}(G)|$ many such isomorphisms each with the same value of $\sgn(\phi_E)$, and if $G$ has non-alternating automorphisms, then $\sum_{\phi \in \on{Iso}(G,G.\sigma)} \frac{\sgn(\phi_E)}{|\Aut G|}$ vanishes because $\sgn(\phi_E)$ is negative for exactly half of the isomorphisms from $G$ to $G. \sigma$.  Hence the formula for $z_g$ may be rewritten as
  \begin{equation*}
    z_g = \sum_n \sum_{G \in \pi_0(\J_{g,n}^\circ)} \sum_{\sigma \in S_n} \sum_{\phi \in \on{Iso}(G,G.\sigma)} \frac{(-1)^{|E(G)|}}{n!|\Aut_{\J_{g,n}^\circ}(G)|} \sgn(\phi_E)\psi(\sigma).
  \end{equation*}
  The set of pairs $(\sigma,\phi)$ with $\sigma \in S_n$ and $\phi \in \on{Iso}(G,G.\sigma)$ is in bijection with $\Aut_{\J^\circ_{g,n}/S_n}(G)$, since $\tau \in \Aut_{\J^\circ_{g,n}/S_n}(G)$ determines a permutation $\tau_{\vert m}$ of $\{1, \dots, n\}$ by restricting along the injective marking $m: \{1, \dots, n\} \to G$.  Hence we get, by collecting the two inner summations into one,
  \begin{equation*}
    z_g = \sum_n \frac1{n!} \sum_{G \in \pi_0(\J^\circ_{g,n})} \frac1{|\Aut_{\J^\circ_{g,n}}(G)|}
    \sum_{\tau \in \Aut_{\J^\circ_{g,n}/S_n}(G)} (-1)^{|E(G)|} \sgn(\tau_E)\psi(\tau_{\vert m}).
  \end{equation*}
  The forgetful functor $v: \J^\circ_{g,n} \to \J^\circ_{g,n}/S_n$ satisfies $v_*v^* 1 = n!$, since the comma category $(v\downarrow G)$ is equivalent to an $n!$ element set.  That is,  
  $$\sum_{G\in \pi_0(\J^\circ_{g,n})} \frac{1}{|\!\Aut_{\J^\circ_{g,n}}(G)|} = 
  n!\!\sum_{G\in{\pi_0(\J^\circ_{g,n}/S_n)}}\frac{1}{|\!\Aut_{\J^\circ_{g,n}/S_n}(G)|}.$$
Therefore
  \begin{align*}
    z_g & = \sum_n \sum_{G \in \pi_0(\J^\circ_{g,n}/S_n)} \frac1{|\Aut_{\J^\circ_{g,n}/S_n}(G)|}
          \sum_{\tau \in \Aut_{\J^\circ_{g,n}/S_n}(G)} (-1)^{|E(G)|} \sgn(\tau_E)\psi(\tau_{\vert m})\\
    & = \sum_n \int_{\mathrm{Fun}(\ZZ,\J^\circ_{g,n}/S_n)} (-1)^{|E(G)|} \sgn(\tau_E)\psi(\tau_{\vert m}) \in \widehat\Lambda,
  \end{align*}
 by Example~\ref{example:free-loops}. 

  Let us write $\sss_0(G,\tau) = (-1)^{|E(G)|} \sgn(\tau_E)$ and for later use point out that $\sss_0(G,\tau)$ equals $(-1)^{|E(G)/\tau|}$, where $E(G)/\tau$ denotes the set of orbits for the action of $\langle\tau\rangle$ on $E(G)$.  If we also write $\ttt_0(G,\tau) = \psi(\tau_{\vert m})$, we have proved that $z_g$ is the orbisum of the function
  \begin{equation*}
    \sss_0 \cdot \ttt_0: \pi_0\bigg(\coprod_{n = 0}^\infty \mathrm{Fun}(\ZZ,\J^\circ_{g,n}/S_n)\bigg) \to \widehat \Lambda.
  \end{equation*}

  \medskip

  For each object $(G,m)$ of $\J^\circ_{g,n}/S_n$ we may forget the marking $m: \{1, \dots, n\} \to V(G)$: the resulting graph $G$ may not be stable, but it may be stabilized by smoothing all 2-valent vertices, leading to a functor
  \begin{equation*}
    u_n: \J^\circ_{g,n}/S_n \to \J^\circ_g,
  \end{equation*}
  which in turn induces functors $U_n = L(u_n): \mathrm{Fun}(\ZZ,\J^\circ_{g,n}/S_n) \to \mathrm{Fun}(\ZZ,\J^\circ_g)$ as in Lemma~\ref{lem:fiber-free-loop-space}.  Taking coproduct over all $n$, we assemble to two ``forget markings'' functors
  \begin{equation}\label{eq:5}
        u: \coprod_n \J^\circ_{g,n}/S_n \to \J^\circ_{g},\qquad U: \coprod_n \mathrm{Fun}(\ZZ,\J^\circ_{g,n}/S_n) \to \mathrm{Fun}(\ZZ,\J^\circ_{g})
  \end{equation}
  to which we will apply Lemma~\ref{lem:orbi-counting} (see also Remark~\ref{rem:inf}).  

  The groupoid $(u \downarrow G)$ is equivalent to a discrete groupoid, namely the set
  \begin{equation*}
    N(G) = \{h\colon (V(G) \amalg E(G)) \to \ZZ_{\geq 0} \mid \text{$h(v) \leq 1$ for all $v \in V(G)$}\},
  \end{equation*}
  regarded as a discrete groupoid, where the function $h_x$ associated to an object $x = ((G',m) \in \J_{g,n}^\circ/S_n, \phi: u_n(G',m) \to G)$ records the cardinalities of the inverse image of the function $m: \{1, \dots, n\} \to V(G') \to V(G) \amalg E(G)$.  By Lemma~\ref{lem:fiber-free-loop-space}, the groupoid $(U \downarrow (G,\tau))$ may be identified with the set $N(G)^{\langle\tau\rangle}$ of $\tau$-invariant such functions.
  
  For brevity, let us in the rest of this proof write $V = V(G)$ and $E = E(G)$, and let us write $E = E_+ \amalg E_-$ where $E_-$ consists of the edges that are reversed by some power of $\tau$ (i.e., fixed points of $\tau^i_E$ that do not lift to fixed points of $\tau_H^i$ for some $i \in \ZZ$) and $E_+$ are those that are not.  To each $\tau$-invariant function $h \in N(G)$ we may associate two other functions
  \begin{equation}\label{eq:11}
    \begin{aligned}
      a_h: (V \amalg &E_-)/\langle\tau\rangle \to \{0,1\}\\
      b_h: (E_+ \amalg &E_-)/\langle\tau\rangle \to \ZZ_{\geq 0},
    \end{aligned}
  \end{equation}
  defined by the requirement that $h(v) = a_h([v])$ for $v \in V$, $h(e) = b_h([e])$ for $e \in E_+$, and $h(e) = 2b_h([e]) + a_h([e])$ for $e \in E_-$.  This sets up a bijection
  \begin{equation}\label{eq:12}
    \begin{aligned}
      N(G)^{\langle\tau\rangle} &\to \{0,1\} ^{(V \amalg E_-)/\langle\tau\rangle} \times \ZZ_{\geq0} ^{(E_+ \amalg E_-)/\langle\tau\rangle}\\
      h & \mapsto (a_h,b_h).
    \end{aligned}
  \end{equation}
  Subdividing $b([e])$ many times each edge in the $\tau_E$-orbit of an $e \in E_+$ will change the number of $\tau_E$-orbits by precisely $b([e])$, while subdividing $2b([e]) + a([e])$ many times each edge in the $\tau_E$-orbit of an $e \in E_-$ changes the number of $\tau_E$-orbits by $b([e])$.  Hence the composition
  \begin{equation*}
    \{0,1\} ^{(V\amalg E_-)/\langle\tau\rangle} \times \ZZ_{\geq 0} ^{(E_+ \amalg E_-)/\langle\tau\rangle} \to     N(G)^{\langle\tau\rangle} \simeq \, (U \downarrow G) \to \coprod_n \J_{g,n}^\circ/S_n \xrightarrow{\alpha_0} \ZZ^\times
  \end{equation*}
  agrees with
  \begin{equation*}
    (a,b) \mapsto \sss_0(G,\tau) \prod_{x \in E/\tau} (-1)^{b(x)}.
  \end{equation*}
  By a similar argument, the composition 
  \begin{equation*}
        \{0,1\} ^{(V\amalg E_-)/\langle\tau\rangle} \times \ZZ_{\geq 0} ^{(E_+ \amalg E_-)/\langle\tau\rangle} \to     N(G)^{\langle\tau\rangle} \simeq \, (U \downarrow G) \to \coprod_n \J_{g,n}^\circ/S_n
        \xrightarrow{\beta_0} \widehat \Lambda
  \end{equation*}
  may be written
  \begin{equation*}
    (a,b) \mapsto \prod_{x \in (V\amalg E_-)/\langle\tau\rangle} p_{|x|}^{a(x)} \prod_{x\in E_+/\tau} p_{|x|}^{b(x)} \prod_{x\in E_-/\tau} p_{2|x|}^{b(x)},
  \end{equation*}
 where we have written $|x|$ for the size of a $\tau$-orbit $x$, when regarded as a subset $x \subset V \amalg E$.   Hence we get
  \begin{equation*}
    (U_*(\sss_0 \cdot \ttt_0))(G,\tau) = \sss_0(G,\tau) \sum_{(a,b)} \prod_{x \in  (V\amalg E_-)/\langle\tau\rangle} p_{|x|}^{a(x)} \prod_{x\in E_+/\tau} (-p_{|x|})^{b(x)} \prod_{x\in E_-/\tau} (-p_{2|x|})^{b(x)}.
  \end{equation*}
   By collecting terms, we get
  \begin{align*}
    (U_*(\sss_0 \cdot \ttt_0))(G,\tau) &= \sss_0(G,\tau) \prod_{x \in (V\amalg E_-)/\tau}(1 + p_{|x|}) \prod_{x \in E_+/\tau}\sum_{i = 0}^\infty (-p_{|x|})^i \prod_{x \in E_-/\tau}\sum_{i = 0}^\infty (-p_{2|x|})^i\\
    &= \sss_0(G,\tau) \frac{\prod_{x \in (V \amalg E_-)/\tau}(1 + p_{|x|})}{\prod_{x \in E_+/\tau}(1 +p_{|x|}) \prod_{x \in E_-/\tau}(1 + p_{2|x|})}.
  \end{align*}
  After multiplying denominator and numerator by $\prod_{x \in E_+/\tau} (1 + p_{|x|})$, we recognize this as
  \begin{equation*}
    (U_*(\sss_0 \cdot \ttt_0))(G,\tau) = \sss_0(G,\tau) \frac{P(\tau_V) P(\tau_E)}{P(\tau_H)}.
  \end{equation*}

  Now, Lemma~\ref{lem:orbi-counting} and Example~\ref{example:free-loops} imply
  \begin{equation*}
    z_g = \sum_{G \in \pi_0(\J_g^\circ)} \frac1{|\Aut(G)|} \sum_{\tau \in \Aut(G)} U_*(\sss_0 \cdot \ttt_0),
  \end{equation*}
  which finishes the proof of Proposition~\ref{prop:start}.
\end{proof}

The conceptual significance of the sets $(V \amalg E_-)/\langle\tau\rangle$ and $(E_+ \amalg E_-)/\langle\tau\rangle$ associated to $(G,\tau)$ is suggested by considering the quotient $|G| \to |G|/\langle\tau\rangle$ by the action of $\langle\tau\rangle$ on the topological space $|G|$.  The quotient is a ``topological graph'' in the sense that there exists a homeomorphism $|G|/\langle\tau\rangle \approx |X|$ for some graph $X$.  We will see next that there is a canonical way to choose $X$ so that there are canonical bijections of sets $V(X) \cong (V \amalg E_-)/\langle\tau\rangle$ and $E(X) \cong (E_+ \amalg E_-)/\langle\tau\rangle$.  In fact a graph $X$ with these properties may be functorially associated to $(G,\tau)$; the next step in our proof of Theorem~\ref{thm:faber-conj} exploits this.

\section{From graphs with automorphisms to orbigraphs} \label{sec:orbigraphs}

Let $G \in \J_g^\circ$ be an object, $\tau\in \Aut_{\J_g^\circ}(G)$, and let $E_-(G) \subset E(G)$ consist of those edges that are reversed by some power of $\tau$ (as in the proof of Proposition~\ref{prop:start}).  Let $G'$ be the graph obtained by barycentrically subdividing each $e \in E_-(G)$ once.  Then we have a $\langle\tau\rangle$-equivariant homeomorphism $|G| \cong |G'|$, but no power of $\tau$ reverses any edge of $G'$.  Define a new graph $X=X(G,\tau)$ by
\begin{equation}\label{eq:8}
  \begin{aligned}
  V(X) = V(G')/\langle\tau\rangle\\
  H(X) = H(G')/\langle\tau\rangle,
\end{aligned}
\end{equation}
and structure maps $s_X$ and $r_X$ induced from those on $G'$.  The fact that no power of $\tau$ reverses any edge of $G'$ implies that $s_X$ is again fixed-point free, and the property $s_X \circ s_X = \mathrm{Id}$ is inherited from $G'$.  Hence $X$ is indeed a graph.  We have canonical homeomorphisms
\begin{equation*}
  |G|/\langle\tau \rangle \approx |G'| /\langle\tau\rangle \approx |X|.
\end{equation*}
We shall also need the function
\begin{equation}\label{eq:7}
  f: V(X) \amalg E(X) \to \ZZ_{>0}
\end{equation}
that measures cardinality of inverse image under the quotient map
\begin{equation*}
  (V(G') \amalg E(G')) \to (V(G') \amalg E(G'))/\langle\tau\rangle = V(X) \amalg E(X).
\end{equation*}
It satisfies $f(r(x)) | f([x])$ for all $x \in H(X)$. 
The pair $(X,f)$ is an example of an orbigraph, defined as follows.
\defnow{
  An \emph{orbigraph} is a pair $(X,f)$ where $X$ is a graph and $f: V(X) \amalg E(X) \to \ZZ_{>0}$ is a function satisfying $f(r(x)) | f([x])$ for all $x \in H(X)$.  An isomorphism of orbigraphs $(X,f) \to (X',f')$ consists of bijections $V(X) \to V(X')$ and $H(X) \to H(X')$ compatible with all structure maps $r$, $s$, and $f$.
}

\newcommand{\OG}{\mathrm{OG}}

\defnow{
  An orbigraph $(X,f)$ is \emph{connected} if $X$ is connected, has \emph{genus} $g$ if $1-g = \sum_{x \in V(X)} f(x) - \sum_{x \in E(X)} f(x)$, and is \emph{stable} if it satisfies
  \begin{enumerate}[(i)]
  \item $\mathrm{val}_X(v) > 0$ for all $v \in V(X)$,
  \item if $\mathrm{val}_X(v) < 3$ then there exists an $h \in H(X)$ with $r(h) = v$ and $f([h]) > f(v)$.
  \end{enumerate}
  Let $\OG_g$ be the groupoid whose objects are stable, connected, genus $g$ orbigraphs, and whose morphisms are the isomorphisms of orbigraphs.
}

\exnow{
  If $(G,\tau) \in \mathrm{Fun}(\ZZ,\J_g^\circ)$, then the orbigraph $(X,f)$ defined by~(\ref{eq:8}) and~(\ref{eq:7}) is a stable connected orbigraph of genus $g$.  We shall denote it by $\mathcal{O}(G,\tau)$.  In this way we have defined a functor between groupoids
  \begin{equation*}
    \mathcal{O}: \mathrm{Fun}(\ZZ,\J_g^\circ) \to \OG_g
  \end{equation*}
}

The number $\chi(X,f) = \sum_{x \in V(X)} f(x) - \sum_{x \in E(X)} f(x)$ is the Euler characteristic of $G$ when $(X,f) = \mathcal{O}(G,\tau)$.  We emphasize that $\chi(X,f)$ is of course usually different from the Euler characteristic of the underlying graph $X$, which we shall denote $\chi(X)$.

\lemnow{\label{lem:function-on-orbigraphs}
  Let $(X,f) = \mathcal{O}(G,\tau)$ be the orbigraph associated to $(G,\tau) \in \mathrm{Fun}(\ZZ,\J_g^\circ)$.  Then
  \begin{equation*}
    (-1)^{|E(G)|} \sgn(\tau_{E(G)}) = (-1)^{|E(X)|}
  \end{equation*}
  and
  \begin{equation*}
    \frac{P(\tau_V) P(\tau_E)}{P(\tau_H)} = \frac{\prod_{x \in V(X)} P_{f(x)}}{\prod_{x \in E(X)}P_{f(x)}} = \prod_{d = 1}^\infty P_d^{\chi(X_d)},
  \end{equation*}
  where $X_d = f^{-1}(d) \subset V(X) \amalg E(X)$ and $\chi(X_d) = |X_d \cap V(X)| - |X_d \cap E(X)|$.
}
\begin{proof}
We have previously noted that $(-1)^{|E(G)|}\sgn(\tau_E) = (-1)^{|E(G)/\langle\tau\rangle|}$. 
But the cardinality of $E(G)/\langle\tau\rangle$ equals that of $E(X)$, establishing the sign.

  The formula for $\frac{P(\tau_V) P(\tau_E)}{P(\tau_H)}$ was already established as part of the proof of Proposition~\ref{prop:start}, since $V(X) = (V \amalg E_-)/\langle\tau\rangle$ and $E(X) = (E_+ \amalg E_-)/\langle\tau\rangle$, in the notation of that proof.
\end{proof}

\defnow{
  For $(X,f) \in \OG_g$, let $\mathcal{Q}(X,f) = (\mathcal{O} \downarrow(X,f))$ denote the comma category whose objects are pairs consisting of a $(G,\tau) \in \mathrm{Fun}(\ZZ,\J_g^\circ)$ and an isomorphism $\phi: \mathcal{O}(G,\tau) \cong (X,f)$.
}

\cornow{\label{cor:s-t-u}
  Let $\sss, \uuu: \pi_0(\OG_g) \to \QQ$, and $\ttt: \pi_0(\OG_g) \to \widehat \Lambda$ be the functions
  \begin{align*}
    \sss(X,f) &= (-1)^{|E(X)|}\\
    \ttt(X,f) &= \prod_{d = 1}^\infty P_d^{\chi(X_d)}\\
    \uuu(X,f) &= (\mathcal{O}_* 1)(X,f) = \sum_{x \in \pi_0(\mathcal{Q}(X,f))} \frac1{|\Aut(x)|}.
  \end{align*}
  Then $z_g = \int_{\OG_g} \sss \cdot \ttt \cdot \uuu \in \widehat\Lambda$, in the notation of \S\ref{sec:orbi-counting-adding}.
}
\begin{proof}
  We already saw that $z_g = \int_{\mathrm{Fun}(\ZZ,\J_g^\circ)} \sss_0 \cdot \ttt_0$, and that $\sss_0 = \mathcal{O}^* \sss$ and $\ttt_0 = \mathcal{O}^* \ttt$.  The claim now follows from (\ref{eq:13}).
\end{proof}

For later use, we give the following description of $\mathcal{Q}(X,f)$, although it will only be useful after further simplifications.  A $\ZZ$-set is a finite set $S$ with a specified action of $\ZZ$; or, equivalently, with a specified bijection $S \to S$.

\defnow{
Let $\mathrm{Orb}$ denote the category whose objects are finite sets with a specified transitive action of $\ZZ$, and whose morphisms are $\ZZ$-equivariant set maps.
}

\lemnow{\label{lem:comma-cate-orbigraph}
  For $(X,f) \in \OG_g$, let $\mathsf{P}(X)$ be the category whose object set is $V(X) \amalg E(X)$, and whose non-identity morphisms are in bijection with $H(X)$, where $x \in H(X)$ is regarded as a morphism from $[x] \in E(X)$ to $r(x) \in V(X)$.  (If $X$ has no loops, this is the ``poset of simplices'' in $X$, ordered by reverse inclusion).  An identification $\phi: \mathcal{O}(G,\tau) \cong (X,f)$ gives rise to a map $V(G') \amalg E(G') \to V(X) \amalg E(X)$ which by abuse of notation we shall also denote $\phi$.  To such an identification we associate a functor
  \begin{equation*}
    j: \mathsf{P}(X) \to \mathrm{Orb}
  \end{equation*}
  by sending $x \in V(X) \amalg E(X)$ to $\phi^{-1}(x) \subset V(G') \amalg E(G')$, regarded as a $\ZZ$-set by the action of $\tau$.  

  This association defines a fully faithful functor from $\mathcal{Q}(X,f)$ to the groupoid 
    \begin{equation}\label{eq:fun-sd-orb}
    \mathrm{Fun}_f(\mathsf{P}(X),\mathrm{Orb})^\sim
    \end{equation}
    whose objects are functors $j$ from $\mathsf{P}(X)$ to $\mathrm{Orb}$, satisfying $|j(x)| = f(x)$ for all $x \in V(X) \amalg E(X)$, and whose morphisms are natural isomorphisms of such functors.  
    }
\noindent The notation $\sim$ records that the morphisms of the category are natural isomorphisms as opposed to all natural transformations.

Figure~\ref{fig:cover} illustrates a genus 2 graph $G$ with an involution $\tau$ reversing each of the two loops, the associated orbigraph $\mathcal{O}(G,\tau)$, as well as the quotient map $|G| \to |G|/\langle\tau\rangle \cong |X|$.  In this case the category $\mathsf{P}(X)$ has five objects and looks like $(\bullet \leftarrow \bullet \rightarrow \bullet \leftarrow \bullet \rightarrow \bullet)$.  The inverse image in $|G|$ of $x \in |X| \cong |G|/\langle\tau\rangle$ is canonically identified with the value of the functor $j$ on the vertex or edge of $X$ given by $x$.  Informally, the lemma asserts that if we know all inverse images of $x \in |X|$ as $\ZZ$-sets, not just their cardinalities, then we may reconstruct $G$ up to canonical isomorphism of graphs.
\begin{figure}[h!]
\begin{tikzpicture}[my_node/.style={fill, circle, inner sep=1.75pt}, scale=1]
\begin{scope}[shift = {(0,0)}]
\draw (1.75,0) node {$\scriptstyle G$};
\node[my_node] (A) at (0,0){};
\draw[ultra thick] (A) to [out = 150, in = 210, looseness=30] (A);
\draw[ultra thick] (A) to [out = 30, in = 330, looseness=30] (A);
\end{scope}
\draw[->] (0,-0.35) to (0,-.85);
\begin{scope}[shift = {(0,-1.5)}]
\node[my_node, label=left:$\scriptstyle 1$] (A) at (-.75,0){};
\node[my_node, label=90:$\scriptstyle 1$] (B) at (0,0){};
\node[my_node, label=right:$\scriptstyle 1$] (C) at (.75,0){};
\draw (1.75,0) node {$\scriptstyle (X,f)$};
\draw[thin] (A)--(C);
\draw (0.37,-0.2) node {$\scriptstyle 2$};
\draw (-0.37,-0.2) node {$\scriptstyle 2$};
\end{scope}
\end{tikzpicture}
\caption{The orbigraph $(X,f) = \mathcal{O}(G,\tau)$, where $\tau$ flips both loops of $G$.}
\label{fig:cover}
\end{figure}
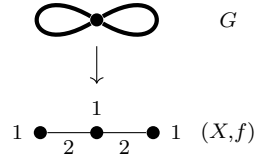

\begin{proof}
  The subdivided graph $G'$ from the definition of $\mathcal{O}(G,\tau)$ may be reconstructed from $j$ as
  \begin{align*}
    V(G') &= \coprod_{y \in V(X)} j(y)\\
    H(G') &= \coprod_{x \in H(X)} j([x]).
  \end{align*}
  The structure map $r: H(G') \to V(G')$ comes from functoriality of $j$, while $s = s_{G'}: H(G') \to H(G')$ uses some extra structure on $\mathsf{P}(X)$.  Namely, the map $s_X: H(X) \to H(X)$ associates to each $f: x \to y$ a morphism $s(f): x \to y'$ with the same source, which induces $s_{G'}: H(G') \to H(G')$.

  It is easily verified that this recipe gives a graph $G'$, with $\ZZ$-action induced from the actions on each $j(x)$, and a canonical isomorphism $G'/\langle\tau\rangle \cong X$ for any functor $j: \mathsf{P}(X) \to \mathrm{Orb}$ with the cardinality restriction as above.  The graph $G$ obtained from $G'$ by smoothing 2-valent vertices will be an object of $\J_g^\circ$ provided it is connected, and in this case we have produced an element of $\mathcal{Q}(X,f)$.  It is easily checked that this process gives an inverse functor to the one described in the lemma.
\end{proof}
The correspondence between $\mathcal{Q}(X,f)$ and the functors $j: \mathsf{P}(X) \to \mathrm{Orb}$ satisfying the cardinality condition is not essentially surjective, because some functors correspond to disconnected $G$.  In fact, we can say exactly what the essential image is.
\lemnow{\label{lem:essential-image}
  The above correspondence gives an equivalence of groupoids to the functors $j$ satisfying in addition that $\mathrm{colim}_{\mathsf{P}(X)} j$ is a singleton.
}
\begin{proof}
  Let $G'$ be the graph associated to the functor $j$, as in the above proof.  Then $E(G')$ may be identified with $\amalg_{x \in E(X)} j(x)$.  Up to canonical homeomorphism, the topological space $|G'| \cong |G|$ is obtained from $V(G') \amalg E(G')$ by gluing an edge for each element of $H(G')$.  Hence $\pi_0(|G|)$ is written as the coequalizer of two particular maps $H(G') \double V(G') \amalg E(G')$, which then by definition is $\mathrm{colim}_{\mathsf{P}(X)} j$.
\end{proof}

A useful condition under which $\cQ(X,f)$ and $\cQ(X',f')$ are equivalent categories is as follows.

\lemnow{\label{lem:half-open}
Let $(X,f) \in \mathrm{OG}_g$, $v$ a valence 2 vertex of $X$ with $r^{-1}(v) = \{h,h'\}$. Let $e=[h]$, $e'=[h']$, and $w = r(s(h))$, and suppose $f(e) = f(v)$.  Note the hypotheses imply that $e\ne e'$.

Let $(X',f')\in \mathrm{OG}_g$ be obtained by setting $r(h') = w$, deleting $v$ and $e$, and restricting $f$ to $(V(X)\amalg E(X)) \setminus \{v,e\}$. Then  $\cQ(X,f)$ and $\cQ(X',f')$ are equivalent groupoids.  See Figure~\ref{fig:half-interval}.
}

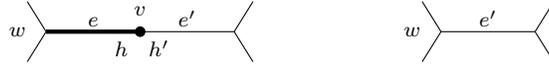
\begin{figure}[h]
\begin{tikzpicture}[v/.style={fill, circle, inner sep=0pt, minimum size=4pt}]
\begin{scope}[shift={(0,0)}, xscale=1.25,yscale=1]
\node[label=left:$\scriptstyle w$]  at (0,0){};
\draw[ultra thick] (0,0)  to (1,0);
\node[v, label=90:$\scriptstyle v$] at (1,0){};
\draw[thin] (0,0)  to (-0.2,0.4);
\draw[thin] (0,0)  to (-0.2,-0.4);
\node[label=90:$\scriptstyle e$] at (0.5,-.2){};
\node[label=90:$\scriptstyle e'$]at (1.5,-.2){};
\node[label=$\scriptstyle h$] at (0.8,-.6){};
\node[label=$\scriptstyle h'$] at (1.2,-.6){};
\draw[thin] (2,0)  to (1,0);
\draw[thin] (2,0)  to (2.2,0.4);
\draw[thin] (2,0)  to (2.2,-0.4);
\end{scope}
\begin{scope}[shift={(4,0)}, xscale=1.25,yscale=1]
\node[label=left:$\scriptstyle w$] (V') at (1,0){};
\node[label=90:$\scriptstyle e'$] (E') at (1.5,-.2){};
\draw[thin] (2,0)  to (1,0);
\draw[thin] (2,0)  to (2.2,0.4);
\draw[thin] (2,0)  to (2.2,-0.4);
\draw[thin] (1,0)  to (.8,0.4);
\draw[thin] (1,0)  to (.8,-0.4);
\end{scope}
\end{tikzpicture}
\caption{Figure accompanying Lemma~\ref{lem:half-open}, $X$ on the left and $X'$ on the right.}\label{fig:half-interval}
\end{figure}

\begin{proof}Define 
$$F\colon \on{Fun}_f(\mathsf{P}(X),\mathrm{Orb})^\sim \to \on{Fun}_{f'}(\mathsf{P}(X'),\mathrm{Orb})^\sim$$
as follows. For $j\in \on{Fun}_f(\mathsf{P}(X),\mathrm{Orb})^\sim$, we obtain $j'=F(j)$ from $j$ setting $j'(h') = j(s(h))\circ j(h)^{-1} \circ j(h')$ and restricting $j$ otherwise.  We used that a morphism $S \to S'$ of finite transitive $\ZZ$-sets is an isomorphism if and only if $|S| = |S'|$; therefore $j(h)$ is an isomorphism.

Then $j$ can be recovered from $j'$, up to natural isomorphism, by setting $j(e) \overset{j(h)}{=\joinrel=} j(v) = j'(e')/f(v)\ZZ$, using the canonical factorization $j'(e) \to j'(e')/f(v)\ZZ \to j'(w)$ of the morphism $j'(h')$ to define the two morphisms $j(h')$ and $j(s(h))$. Such a factorization exists since $f(w)|f(v)$.  Moreover $\mathrm{colim}_{\mathsf{P}(X)} j = \mathrm{colim}_{\mathsf{P}(X')} j'$ and in particular one is a singleton $\ZZ$-set if and only if the other is. Thus $F$ restricts to an equivalence of categories between $\cQ(X,f)$ and $\cQ(X',f')$.  
\end{proof}

\section{From orbigraphs to static orbigraphs}  \label{sec:static}

We now introduce a notion of exhalations and inhalations for orbigraphs, and show that the total contribution to $z_g$ coming from orbigraphs that admit a nontrivial exhalation or inhalation is zero.  This will reduce our computation of $z_g$ to a sum over \emph{static orbigraphs}, i.e., those that admit no nontrivial exhalations or inhalations.

\defnow{Let $(X,f) \in \OG_g$.  An edge $e \in E(X)$ is \emph{exhalable} if its two endpoints $v, v' \in V(X)$ are distinct, if $f(v) = f(v') = f(e)$, and if at least one of $v,v'$ has valence $2$ and the other edge $e'$ at that vertex has $f(e') > f(e)$.  Let $\mathrm{Exh}(X,f)\subset E(X)$ be the set of exhalable edges.

  If $e \in \mathrm{Exh}(X,f)$ we define the \emph{exhalation} to be the orbigraph $(X',f')$ obtained by collapsing $e$ to a new vertex $\tilde{v}$ and setting $f(\tilde{v}) = f(e)$. See Figure~\ref{fig:inex}.  
}

\begin{figure}[h!]
\begin{tikzpicture}[v/.style={fill, circle, inner sep=0pt, minimum size=4pt}]
\begin{scope}[shift={(4.5,0)}]
\node[v, label=left:$\scriptstyle \tilde v$] (V) at (0,0){};
\draw[thin] (0,0) to (0.75,0);
\draw[thin] (0,0)  to (-0.2,0.4);
\draw[thin] (0,0)  to (-0.2,-0.4);
\node[label=90:$ $] (E) at (0.5,-0.2){};
\end{scope}
\begin{scope}[shift={(0,0)}]
\node[v, label=left:$\scriptstyle v$] (V) at (0,0){};
\draw[ultra thick] (0,0)  to (1,0);
\node[v, label=90:$\scriptstyle v'$] (V') at (1,0){};
\draw[thin] (1,0) to (1.75,0);
\draw[thin] (0,0)  to (-0.2,0.4);
\draw[thin] (0,0)  to (-0.2,-0.4);
\node[label=90:$\scriptstyle e$] (E') at (0.5,-0.2){};
\node[label=90:$ $] (E) at (1.5,-0.2){};
\end{scope}
\draw[<->] (2.5,0) to (3,0);
\draw[<->] (2.5,-2) to (3,-2);
\begin{scope}[shift={(4.5,-2)}]
\node[v, label=90:$\scriptstyle \tilde v$] (V) at (0,0){};
\draw[thin] (0,0) to (0.5,0);
\draw[thin] (0,0)  to (-0.5,0);
\end{scope}
\begin{scope}[shift={(0,-2)}]
\node[v, label=90:$\scriptstyle v$] (V) at (0,0){};
\draw[ultra thick] (0,0)  to (1,0);
\node[v, label=90:$\scriptstyle v'$] (V') at (1,0){};
\draw[thin] (1,0) to (1.5,0);
\draw[thin] (0,0)  to (-0.5,0);
\node[label=90:$\scriptstyle e$] (E') at (0.5,-0.2){};
\node[label=90:$ $] (E) at (1.5,-0.2){};
\end{scope}

\end{tikzpicture}
\caption{Exhaling, from left to right. Inhaling, from right to left.}\label{fig:inex}
\end{figure}
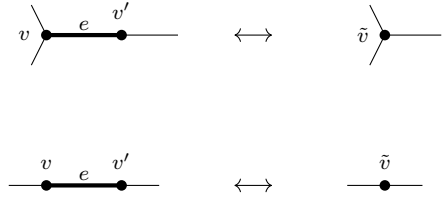

\lemnow{
  Let $(X',f')$ be the exhalation of $(X,f) \in \OG_g$ along some $e \in \mathrm{Exh}(X,f)$.  Then
  \begin{equation*}
    \mathrm{Exh}(X',f') = \mathrm{Exh}(X,f) \setminus \{e\},
  \end{equation*}
  and the exhalation of $(X',f')$ along some $e' \in \mathrm{Exh}(X',f')$ is canonically isomorphic to the exhalation performed in the other order.\qed
}
\defnow{
  For an orbigraph $(X,f)$, let $\mathrm{Ex}(X,f)$ be the orbigraph obtained by iterated exhalation of $(X,f)$ along all exhalable edges.  This defines a functor
  \begin{equation*}
    \mathrm{Ex}: \OG_g \to \OG_g
  \end{equation*}
  which we call the \emph{maximal exhalation}.
}
Recall that we wish to calculate $z_g = \int_{\OG_g} \sss \cdot \ttt \cdot \uuu$, in the notation of Corollary~\ref{cor:s-t-u}.  As we shall see, the function $ \sss \cdot \ttt \cdot \uuu$ simplifies drastically by pushing it forward along $\mathrm{Ex}$.  To study the comma categories $(\mathrm{Ex} \downarrow (X,f))$, we note that exhaling can be reversed:

\defnow{Let $(X,f)$ be an orbigraph.  
  \begin{enumerate}[(i)]
  \item An element $v \in V(X)$ is \emph{inhalable} if it has valence 2 in $X$ and $f(h) > f(v)$ for both half-edges $h \in H(X)$ incident to $v$.
  \item An element $h \in H(X)$ is inhalable if $v$ has valence at least 3 and if $f(h) > f(v)$, for the vertex $v = r(h)$ incident to $h$.
  \end{enumerate}
  In either case the \emph{inhalation} of $(X,f)$ along $h$ or $v$
  is the orbigraph $(X',f')$ obtained by expanding $v$ into an edge $e = vv'$, with $v'$ incident to $h$ and $v$ incident to the other half-edges at $v$, and extending $f$ to $f': V(X') \amalg E(X') \to \ZZ_{>0}$ by setting $f'(e) = f'(v')=f(v)$.  See again Figure~\ref{fig:inex}.

  Let $\mathrm{Inh}(X,f) \subset V(X) \amalg H(X)$ be the set of inhalable elements.
}

\lemnow{\label{lem:fiber-of-max-exh}
  Let $(X,f) \in \OG_g$.  The groupoid $(\mathrm{Ex}\downarrow (X,f))$  is empty unless $\mathrm{Exh}(X,f) = \emptyset$, in which case it is equivalent to the power set of $\mathrm{Inh}(X,f)$, regarded as a discrete category: the equivalence is defined by sending a subset to the orbigraph obtained by iterated inhalation along those elements.
}

See Figure~\ref{fig:Boolean} for an example where $|\mathrm{Inh}(X,f)| = 3$ so $(\mathrm{Ex} \downarrow (X,f))$ is equivalent to a set with $2^3 = 8$ elements, regarded as a discrete groupoid.

\begin{figure}
\begin{tikzpicture}[v/.style={fill, circle, inner sep=0pt, minimum size=4pt}]
\begin{scope}[shift={(0,0)}, scale=1]
\node[v, label=$ $] (a) at (0,0){};
\node[label=180:$\scriptstyle 1$] at (.2,0){};
\node[label=180:$\scriptstyle 1$] at (-.1,.6){};
\node[label=180:$\scriptstyle 1$] at (-.1,-.6){};
\node[label=0:$\scriptstyle 1$] at (.6,0){};
\node[label=90:$\scriptstyle 2$] at (.4,-.2){};
\node[label=0:$\scriptstyle 2$] at (-.4,.4){};
\node[label=0:$\scriptstyle 2$] at (-.4,-.4){};
\node[v] (b) at (.8,0){};
\node[v] (c) at (-0.3,.6){};
\node[v] (d) at (-0.3,-.6){};
\draw[thin] (a) to (b);
\draw[thin] (a) to (c);
\draw[thin] (a) to (d);
\end{scope}
\begin{scope}[shift={(-2,1.5)}, scale=.7]
\node[v, label=$ $] (a) at (0,0){};
\node[label=180:$ $] at (.2,0){};
\node[v, label=0:$ $] (b) at (.8,0){};
\node[v, label=90:$ $] (c) at (-0.5,1){};
\node[v, label=90:$ $] (c') at (-0.25,.5){};
\node[v, label=270:$ $] (d) at (-0.3,-.6){};
\draw[thin] (a) to (b);
\draw[ultra thick] (a) to (c');
\draw[thin] (a) to (c);
\draw[thin] (a) to (d);
\end{scope}
\begin{scope}[shift={(-.2,1.6)}, scale=.7]
\node[v, label=$ $] (a) at (0,0){};
\node[label=180:$ $] at (.2,0){};
\node[v, label=0:$ $] (b') at (.6,0){};
\node[v, label=0:$ $] (b) at (1.2,0){};
\node[v, label=90:$ $] (c) at (-0.3,.6){};
\node[v, label=270:$ $] (d) at (-0.3,-.6){};
\draw[thin] (a) to (b);
\draw[ultra thick] (a) to (b');
\draw[thin] (a) to (c);
\draw[thin] (a) to (d);
\end{scope}
\begin{scope}[shift={(1.75,1.75)}, scale=.7]
\node[v, label=$ $] (a) at (0,0){};
\node[label=180:$ $] at (.2,0){};
\node[v, label=0:$ $] (b) at (.8,0){};
\node[v, label=90:$ $] (c) at (-0.3,.6){};
\node[v, label=270:$ $] (d) at (-0.5,-1){};
\node[v, label=90:$ $] (d') at (-0.25,-.5){};
\draw[thin] (a) to (b);
\draw[ultra thick] (a) to (d');
\draw[thin] (a) to (c);
\draw[thin] (a) to (d);
\end{scope}
\begin{scope}[shift={(-2,3)}, scale=.7]
\node[v, label=$ $] (a) at (0,0){};
\node[label=180:$ $] at (.2,0){};
\node[v, label=0:$ $] (b') at (.6,0){};
\node[v, label=0:$ $] (b) at (1.2,0){};
\node[v, label=90:$ $] (c) at (-0.5,1){};
\node[v, label=90:$ $] (c') at (-0.25,.5){};
\node[v, label=270:$ $] (d) at (-0.3,-.6){};
\draw[thin] (a) to (b);
\draw[ultra thick] (b') to (a) to (c');
\draw[thin] (a) to (c);
\draw[thin] (a) to (d);
\end{scope}
\begin{scope}[shift={(-.2,3.2)}, scale=.7]
\node[v, label=$ $] (a) at (0,0){};
\node[label=180:$ $] at (.2,0){};
\node[v, label=0:$ $] (b) at (.8,0){};
\node[v, label=90:$ $] (c) at (-0.5,1){};
\node[v, label=90:$ $] (c') at (-0.25,.5){};
\node[v, label=270:$ $] (d) at (-0.5,-1){};
\node[v, label=90:$ $] (d') at (-0.25,-.5){};
\draw[thin] (a) to (b);
\draw[ultra thick] (d') to (a) to (c');
\draw[thin] (a) to (c);
\draw[thin] (a) to (d);
\end{scope}
\begin{scope}[shift={(1.5,3.2)}, scale=.7]
\node[v, label=$ $] (a) at (0,0){};
\node[label=180:$ $] at (.2,0){};
\node[v, label=0:$ $] (b') at (.6,0){};
\node[v, label=0:$ $] (b) at (1.2,0){};
\node[v, label=90:$ $] (c) at (-0.3,.6){};
\node[v, label=270:$ $] (d) at (-0.5,-1){};
\node[v, label=90:$ $] (d') at (-0.25,-.5){};
\draw[thin] (a) to (b);
\draw[ultra thick] (b') to (a) to (d');
\draw[thin] (a) to (c);
\draw[thin] (a) to (d);
\end{scope}
\begin{scope}[shift={(-.2,5)}, scale=.7]
\node[v, label=$ $] (a) at (0,0){};
\node[label=180:$ $] at (.2,0){};
\node[v, label=0:$ $] (b') at (.6,0){};
\node[v, label=0:$ $] (b) at (1.2,0){};
\node[v, label=90:$ $] (c) at (-0.5,1){};
\node[v, label=90:$ $] (c') at (-0.25,.5){};
\node[v, label=270:$ $] (d) at (-0.5,-1){};
\node[v, label=90:$ $] (d') at (-0.25,-.5){};
\draw[thin] (a) to (b);
\draw[ultra thick] (d') to (a) to (c');
\draw[ultra thick] (b') to (a);
\draw[thin] (a) to (c);
\draw[thin] (a) to (d);
\end{scope}
\end{tikzpicture}
\caption{An instance of $(\mathrm{Ex} \downarrow (X,f))$}
\label{fig:Boolean}
\end{figure}
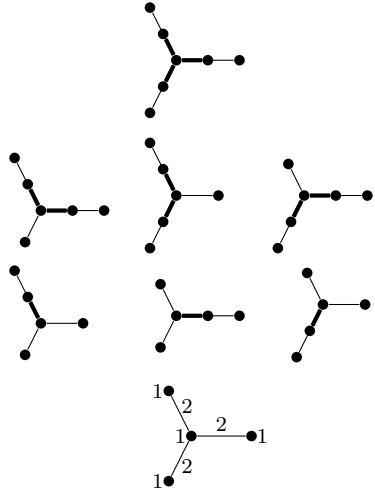

\begin{proof}
  An object of $(\mathrm{Ex}\downarrow (X,f))$ consists of an $(X',f') \in \OG_g$ and an isomorphism $\mathrm{Ex}(X',f') \cong (X,f)$.  Each $e \in \mathrm{Exh}(X,f)$  is collapsed to a vertex $v'$ in $X$.  If the valence of $v'$ is at least 3, it comes with a distinguished incident half-edge $h$, corresponding to the 2-valent vertex of $e$.  Hence we associate to $e$ an element of $\mathrm{Inh}(X,f)$, either $v'$, if it is 2-valent, or $h$.  By inhaling all these elements, we get back an orbigraph with a canonical isomorphism to $(X',f')$: the identity map of the non exhaled/inhaled elements of $H(X',f') \amalg V(X',f')$ extends uniquely to an isomorphism.
\end{proof}

Note the following special case of Lemma~\ref{lem:half-open}.

\lemnow{\label{lem:homework}
  For any orbigraph $(X,f)$, there is a (canonical) equivalence of categories
  \begin{equation*}
    \mathcal{Q}(X,f) \cong \mathcal{Q}(\mathrm{Ex}(X,f)).
  \end{equation*}
}

\cornow{
  Let $\sss$, $\ttt$, and $\uuu$ be as in Corollary~\ref{cor:s-t-u}.  Then
  \begin{equation*}
    (\mathrm{Ex}_*( \sss \cdot \ttt \cdot \uuu))(X,f) = 0
  \end{equation*}
  if $(X,f)$ admits an inhalation or an exhalation.
}
\begin{proof}
  If $(X,f)$ admits an exhalation then it is not in the essential image of $\mathrm{Ex}$.  Hence that case is obvious, and we assume $(X,f) = \mathrm{Ex}(X,f)$.
  
  Lemma~\ref{lem:homework} shows that $\uuu$ is invariant under exhalations, and it is easy to check that $\ttt$ is also invariant under exhalations: this is because $\ttt(X,f)$ depends only on the numbers $\chi(X_d)$ and these are preserved by exhalations.  Hence $\ttt \cdot \uuu = \mathrm{Ex}^*(\ttt \cdot\uuu)$, so by~\eqref{eq:13}
  \begin{equation*}
    (\mathrm{Ex}_*( \sss \cdot \ttt \cdot \uuu)) = (\ttt \cdot \uuu) \cdot (\mathrm{Ex}_*\sss). 
  \end{equation*}
  The function $\sss$ is not preserved by inhalations, and in fact it changes sign once for each inhalation.  Therefore Lemma~\ref{lem:fiber-of-max-exh} implies that
  \begin{equation*}
    (\mathrm{Ex}_* \sss)(X,f) = \sss(X,f) \sum_{S \subset \mathrm{Inh}(X,f)} (-1)^{|S|},
  \end{equation*}
  which is zero when $\mathrm{Inh}(X,f) \neq \emptyset$.
\end{proof}

\defnow{
  An orbigraph $(X,f)$ is \emph{static} if it admits neither
  exhalations nor inhalations.  Let $\OG_g^\mathrm{stat} \subset \OG_g$ be the full subcategory consisting of static orbigraphs.}
\cornow{\label{cor:reduction-to-static}
  Let $\sss$, $\ttt$, and $\uuu$ be as in Corollary~\ref{cor:s-t-u}.  Then
  \begin{equation*}
    z_g = \int_{\OG_g^\mathrm{stat}}  \sss \cdot \ttt \cdot \uuu.
  \end{equation*}
}

\section{Static orbigraphs}\label{sec:static-structure}

We now give a structural classification of static orbigraphs.

\defnow{
  A \emph{tail} of an orbigraph $(X,f)$ is a finite sequence $(h_0, \dots, h_k)$ of elements of $H(X)$ with $k \in \ZZ_{\ge 0}$, such that
  \begin{enumerate}
  \item $v_0 = r(h_0)$ is 1-valent, and
  \item $v_i = r(h_{i})$ is 2-valent and equals $r \circ s(h_{i-1})$, for $i = 1, \dots, k$, and
  \item $f(v_i) = f(e_{i-1})$, where $e_i = \{h_i,s(h_i)\} \in E(X)$, for $i = 0, \dots, k-1$.
  \end{enumerate}
  In that case we write $d = f(v_0)$ and $x_i = f(v_i)$ for $i = 1, \dots, k$.  We necessarily have $d | x_1 | \dots | x_k$.

  The \emph{length} of the tail is the number $k$.  The tail is \emph{maximal} if it cannot be extended to a longer tail $(h_0, \dots, h_k, h_{k+1})$.
}

It is easily seen that any tail is part of a unique maximal tail, and that distinct maximal tails $(h_0, \dots, h_k)$ and $(h'_0, \dots, h'_{k'})$ are disjoint. 

\defnow{
  The \emph{cropping} of a tail $(h_0,\dots, h_k)$ of an orbigraph $(X,f)$ is the orbigraph $(X',f')$ obtained by removing $\{v_0, \dots, v_{k-1}\}$ from $V(X)$ and $\{h_0, s(h_0), \dots, h_{k-1}, s(h_{k-1})\}$ from $H(X)$ and defining $f'$ by setting $f'(v_{k}) = d$, where $d = f(v_0)$, and restricting $f$ otherwise.
}
In particular, cropping a tail of length zero results in an orbigraph canonically isomorphic to the input.  Cropping a maximal tail leaves a length-zero tail, ending in a 1-valent vertex.  
The original orbigraph $(X,f)$ may be reconstructed up to canonical isomorphism from the data of $(X',f')$, the marked point $v_k \in V(X')$, and the sequence $d | x_1 | \dots | x_k = f'(v_k)$.  Notice that cropping preserves the number $\chi(X,f) = \sum_{x \in V(X)} f(x) - \sum_{x \in E(X)} f(x)$.

Lemma~\ref{lem:half-open}, and induction on length of tail, immediately give the following; compare with Lemma~\ref{lem:homework} for exhalation.

\lemnow{\label{lem:crop}
  If $(X',f')$ is obtained from $(X,f)$ by cropping a tail, then there is an equivalence of groupoids
  \begin{equation*}
    \mathcal{Q}(X,f) \simeq \mathcal{Q}(X',f').
  \end{equation*}
}

\begin{figure}
\begin{tikzpicture}[v/.style={fill, circle, inner sep=0pt, minimum size=0pt}, v2/.style={fill, circle, inner sep=0pt, minimum size=3.5pt}]
\draw[thin] (0,0) circle [radius=1];
\node[v2] (a) at (2,0){};
\node[v2] (a2) at (3,0){};
\node[v2] (a3) at (4,0){};
\draw[thin] (1,0) to (a);
\draw[thick] (a) to (a2);
\draw[thick] (a3) to (a2);
\node[] at (2,0.2) {$\scriptstyle v_2$};
\node[] at (2.5,0.15) {$\scriptstyle v_2$};
\node[] at (3,0.2) {$\scriptstyle v_1$};
\node[] at (3.5,.15) {$\scriptstyle v_1$};
\node[] at (4,0.2) {$\scriptstyle d$};
\node[v2] (b) at (-1.25,1.1){};
\node[v2] (b2) at (-1.75,1.2){};
\draw[thin] (-.707,.707) to (b);
\draw[thick] (b) to (b2);
\node[v] (c) at (-.939,-.342){};
\node[v] (c2) at (-1.4095,-0.5130){};
\node[v2] (c3) at (-2,-.6){};
\node[v2] (c4) at (-1.5,-.9){};
\draw (c) to (c2) to (c3);
\draw (c2) to (c4);
\end{tikzpicture}
\caption{A static orbigraph with maximal tails of lengths $0,0,1,2$.}\label{fig:static-old}
\end{figure}
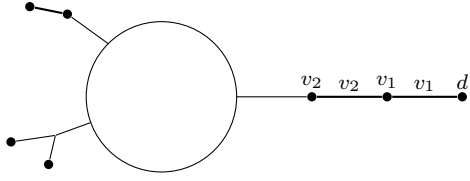

The following proposition classifies static orbigraphs.
\propnow{\label{prop:structure-of-static}
Let $(X,f)$ be an orbigraph, and let $(X',f')$ be the orbigraph obtained by cropping all maximal tails of $(X,f)$.  Then $(X,f)$ is static if and only if $f'$ is constant away from any $1$-valent vertices of $X'$.
}
\proofnow{
Notice that cropping tails does not change the inhalable or exhalable sets.  Therefore we may assume $(X,f) = (X',f')$ has no positive length tail.

If $(X,f)$ admits an inhalation or exhalation, then $f$ cannot be constant away from $1$-valent vertices.
Now suppose $(X,f)$ is static. Consider the set of vertices $v\in V(X)$ of valence at least $2$, such that $f([h]) > f(v)$ for some $h\in H(X)$ with $r(h)=v$. It suffices to show this set is empty. Suppose not, and choose a vertex $v$ in that set, with $f(v)$ minimal.  Now if $v$ has valence at least $3$ then $h$ would be inhalable. Therefore $v$ has valence $2$.  The other half-edge $h'$ at $v$ must have $f([h'])=f(v)$, otherwise $v$ inhales.  Set $v' =r(s(h'))$.  Now if $f(v') = f([h'])$ then $[h']$ exhales. Therefore $f(v')<f([h']) = f(v).$  This contradicts the choice of $v$ unless $v'$ is $1$-valent. But then $(X,f)$ has a positive length tail, contradiction.
}

\defnow{\label{def:reduction-functor}
  A static orbigraph is called \emph{reduced} if all maximal tails have length 0.  Let $\OG_g^{\mathrm{stat},\mathrm{red}} \subset \OG_g^\mathrm{stat}$ be the full subcategory on those objects.  Let
  \begin{equation*}
    \mathcal{R}: \OG_g^{\mathrm{stat}} \to \OG_g^{\mathrm{stat},\mathrm{red}}
  \end{equation*}
  be the functor which sends sends $(X,f)$ to the orbigraph obtained by cropping all maximal tails of $(X,f)$.
}

The following lemma gives a 
decomposition of the groupoid $\OG_g^{\mathrm{stat},\mathrm{red}}$, corresponding to the sum in Theorem~\ref{thm:faber-conj}.

\begin{lemma}\label{lem:stat-red-coprod}
  For $m > 0$ and $r,s \geq 0$, and $d = (d_1, \dots, d_s) \in \ZZ_{>0}^s$ and $a = (a_1, \dots, a_s) \in \ZZ_{>0}^s$, let
  \begin{equation*}
    \OG^{\mathrm{stat},\mathrm{red}}_{(g,m,r,s,a,d)} \subset \OG_g^{\mathrm{stat},\mathrm{red}}
  \end{equation*}
  be the full subgroupoid consisting of static, reduced orbigraphs $(X,f)$ with $\chi(X,f) = 1-g$ and $\chi(X) = 1-r$, such that the maximal value of $f$ is $m$ and the values of $f$ on 1-valent vertices of $(X,f)$ are $d_1 < \dots < d_s$, with multiplicity $a_i = |f^{-1}(d_i)|$.  Then the natural functor
  \begin{equation*}
    \coprod_{(k,m,r,s,a,d)} \OG^{\mathrm{stat},\mathrm{red}}_{(g,m,r,s,a,d)} \to \OG_g^{\mathrm{stat},\mathrm{red}},
  \end{equation*}
  is an equivalence of groupoids, where $(k,m,r,s,a,d)$ runs over all tuples satisfying the conditions~(\ref{eq:di-divides})--(\ref{eq:chi-upstairs}) from Theorem~\ref{thm:faber-conj}.  (The number $k$ is determined by the rest of the data, but its existence is a condition.)
\end{lemma}
\begin{proof}
  It is obvious that the tuple $(m,r,s,a,d)$ is invariant under isomorphisms in $\OG_g^{\mathrm{stat},\mathrm{red}}$, so we get a splitting according to this data.  The number $k$ defined by~(\ref{eq:chi-downstairs}) may be interpreted as minus the Euler characteristic of $X$ relative to the tail points.  Since $|G| \to |G|/\langle\tau\rangle \cong |X|$ is an $m$-fold cover away from the tail points, we get~(\ref{eq:chi-upstairs}) by ``graph-theoretic Riemann--Hurwitz''.
\end{proof}

\section{Contribution from reduced static orbigraphs} \label{sec:static}

We have seen $z_g = \int_{\OG_g^\mathrm{stat}}  \sss \cdot \ttt \cdot \uuu$.  In this section we study some properties of the functions $\sss$, $\ttt$, and $\uuu$ with respect to the functor $\mathcal{R}: \OG_g^\mathrm{stat} \to \OG_g^{\mathrm{stat}, \mathrm{red}}$.
\begin{lemma}\label{lem:t-u-factor}
  For any static orbigraph $(X,f)$ we have
  \begin{align*}
    \ttt(X,f) &= \ttt(\mathcal{R}(X,f))\\
    \uuu(X,f) &= \uuu(\mathcal{R}(X,f)),
  \end{align*}
  i.e., $\uuu = \mathcal{R}^* \uuu$ and $\ttt = \mathcal{R}^* \ttt$, in the notation of \S\ref{sec:orbi-counting-adding}.   
\end{lemma}
\begin{proof}
  The function $\ttt(X,f) = \prod_{d = 1}^\infty P_d^{\chi(X_d)}$ depends only on the numbers $\chi(X_d)$ which are invariant under cropping tails.  The number $\uuu(X,f)$ is the groupoid cardinality of $\mathcal{Q}(X,f)$, but by Lemma~\ref{lem:crop} this is preserved by cropping tails.
\end{proof}

We can also evaluate the functions $\ttt$ and $\uuu$ on reduced static orbigraphs, using the classification in Lemma~\ref{lem:stat-red-coprod}.

\propnow{\label{prop:u}
  Let $(X,f)$ be a static reduced orbigraph whose underlying graph has genus $r$ (i.e., $\chi(X) = 1-r$, while $\chi(X,f) = 1-g$), let $m \in \ZZ_{>0}$ be the (constant) value of $f$ away from the tails, let $\{d_1 < \dots < d_s\}$ be the set of values of $f$ on tails of $(X,f)$, and let $a_i = |f^{-1}(d_i)|$ be the multiplicities of those values.  Then we have
  \begin{equation*}
    \ttt(X,f) = P_m^{1 - r} \prod_{i = 1}^s \left(\frac{P_{d_i}}{P_m}\right)^{a_i}
  \end{equation*}
  and
  \begin{equation*}
    \uuu(X,f) = m^{r-1} \prod_{p | D} (1 - p^{-r}),
  \end{equation*}
  where the product is over all prime numbers $p$ dividing $D = \gcd(m,d_1, \dots, d_s)$.
}
\begin{proof}
  The formula for $\ttt(X,f) \in \widehat\Lambda$ comes directly from the formula in Corollary~\ref{cor:s-t-u}.

  The number $\uuu(X,f)\in\QQ$ is defined as the groupoid cardinality of $\mathcal{Q}(X,f)$, which we identify as in the proof of Lemma~\ref{lem:homework} with a full subcategory of
  \begin{equation}\label{eq:6}
    \mathrm{Fun}_f(\mathsf{P}(X),\mathrm{Orb})^\sim,
  \end{equation}
  the groupoid whose objects are functors $j: \mathsf{P}(X) \to \mathrm{Orb}$ satisfying $|j(x)| = f(x)$ for all $x$, and whose morphisms are natural isomorphisms of such functors.  Recall that in general, $\mathcal{Q}(X,f)$ is only a subgroupoid because objects of $\J_g^\circ$ are required to be connected. As in the proof of Lemma~\ref{lem:crop}, the values of such $j$ at the 1-valent vertices of $X$ are canonically determined by the values at the incident edges, up to unique isomorphism.  Away from the 1-valent vertices $f$ is constantly $m$, so all morphisms must go to isomorphisms between finite transitive $\ZZ$-sets of cardinality $m$.  Therefore the groupoid~(\ref{eq:6}) may be identified with the groupoid of $m$-fold cyclic covering spaces of $|X|$ (i.e., principal $\ZZ/m\ZZ$-bundles over $|X|$; given such a bundle the corresponding functor $j$ is given on objects by sending $x \in V(X) \amalg E(X)$ to the fiber over the corresponding point in $|X|$ and on morphisms by monodromy of the bundle).  After picking a basepoint in $|X|$ and a basis for the free group $\pi_1(|X|)$, the groupoid~(\ref{eq:6}) may therefore be identified with
  \begin{equation*}
    \mathrm{Fun}(F_r,\ZZ/m\ZZ),
  \end{equation*}
  where $F_r$ denotes the free group on $r$ letters and $\ZZ/m\ZZ$ the cyclic group with $m$ elements, regarded as groupoids with one object.  
The objects of $\mathrm{Fun}(F_r,\ZZ/m\ZZ)$ are identified with group homomorphisms $F_r\to \ZZ/m\ZZ$, and given $A,B\colon F_r\to\ZZ/m\ZZ$, there is a natural isomorphism from $A$ to $B$ for every $z\in\ZZ/m\ZZ$ such that $B(x) = zA(x)z^{-1}$ for every $x$.  Since $\ZZ/m\ZZ$ is abelian, $\pi_0(\mathrm{Fun}(F_r,\ZZ/m\ZZ)) = \on{Hom}_{\mathrm{Gp}}(F_r,\ZZ/m\ZZ)$ and the automorphism group of any functor is cyclic of order $m$.

  It remains to understand the connectivity condition: we only want covering spaces $q: P \to |X|$ whose total space $P$ becomes connected after taking the quotient by the action of $\langle\tau^{f(x)}\rangle$ on each fiber $q^{-1}(x)$, as $x \in |X|$ ranges over the 1-valent vertices.  Therefore such covering spaces are in correspondence with the set 
    \begin{equation}\label{eq:10}
    \{(z_1, \dots, z_r) \in (\ZZ/m\ZZ)^r \mid \text{the mod $D$ reductions of the $z_i$ generate $\ZZ/D\ZZ$}\},
  \end{equation}
  which has cardinality
  \begin{equation*}
    m^r \prod_{p | D} (1 - p^{-r}).
  \end{equation*}
The groupoid cardinality of the subgroupoid of~\eqref{eq:6} corresponding to $\mathcal{Q}(X,f)$ is obtained by dividing that quantity by $m$.
\end{proof}

Unlike $\ttt$ and $\uuu$, the number $\sss(X,f) = (-1)^{|E(X)|}$ is not invariant under cropping tails.  Instead we have the following.
\lemnow{\label{lem:s}
  Let $(X,f)$ be a reduced static orbigraph, let $m$ be the maximum value of $f$, and let $\{d_1 < \dots < d_s\}$ be the set of values of $f$ on $1$-valent vertices of $(X,f)$, and let $a_i = |f^{-1}(d_i)|$ be the multiplicities of those values. Then
  \begin{equation*}
    (\mathcal{R}_* \sss)(X,f) = \sss(X,f) \cdot \prod_{i=1}^s (-\mu(m/d_i))^{a_i},
  \end{equation*}
  where $\mu$ is the M\"obius function.
}
\begin{proof}
  The groupoid $(\mathcal{R}\downarrow (X,f))$ may be identified with the product, over $i=1,\ldots,s,$ of the sets of chains between $1$ and $m/d_i$ in the poset (divisors of $m/d_i$, divisibility), regarded as a discrete groupoid.  The function to be summed may be identified with $\sss(X,f)$ times the product of $(-1)^{1 + \text{length of $i$th chain}}$.

  It is well known that the sum, over chains from $1$ to $a$ of divisors of a natural number $a$, of $(-1)^\text{length of chain}$ gives the M\"obius function $\mu(a)$.
\end{proof}

We can also evaluate the number $\sss(X,f)$ for reduced static orbigraphs, by appealing to a calculation of a certain orbifold Euler characteristic, due to Kontsevich.
\begin{lemma}\label{lem:orbifold-euler}
  Let $\OG^{\mathrm{stat},\mathrm{red}}_{(g,m,r,s,a,d)} \subset \OG_g^{\mathrm{stat},\mathrm{red}}$ be as in Lemma~\ref{lem:stat-red-coprod}.  Then
  \begin{equation*}
    \int_{\OG^{\mathrm{stat},\mathrm{red}}_{(g,m,r,s,a,d)}} \sss =
    \frac{-1}{(a_1!) \dots (a_s)!} \frac{(r+n-2)!}{r!} \cdot B_r,
    \end{equation*}
    where $B_r$ is the Bernoulli number and $n=|a|=a_1+\cdots+a_s$.
\end{lemma}
\begin{proof}
First dispense with a special case that $(r,n) = (0,2)$. Then either $s=1$, $a_1=2$, and both sides are $-1/2$, or $s=2$, $a_1=a_2=1$, and both sides are $-1$.  

Otherwise, an object of $\OG^{\mathrm{stat},\mathrm{red}}_{(g,m,r,s,a,d)}$ has $n$ many tails, each of which consists of an edge ending in the tail point, which is a valence-1 vertex.  If we pick a total ordering of the set of the $a_1$ many tails with value $d_1$, a total ordering of the set of the $a_2$ many tails with value $d_2$, etc, we arrive at an object of the groupoid $\J_{r,n}^\mathrm{p}$ defined exactly as $\J_{r,n}^\circ$, except that we do not require the marking function $m: \{1, \dots, n\} \to V$ to be injective.  The superscript $p$ stands for ``pure,'' as in \cite{cgp-markedpoints-published}. 
    Accounting for the choices of orderings of marked points, and the $n$ many extra edges where the tails are attached, we therefore get
\begin{equation*}
  \int_{\OG^{\mathrm{stat},\mathrm{red}}_{(g,m,r,s,a,d)}} \sss = \frac{(-1)^{n}}{(a_1!) \dots (a_s)!} \int_{\J^{\mathrm{p}}_{r,n}} \sss = \frac{(-1)^{n}}{(a_1!) \dots (a_s)!} \sum_{(G,m) \in \J^{\mathrm{p}}_{r,|a|}} \frac{(-1)^{|E(G)|}}{|\Aut(G)|}.
\end{equation*}
The lemma now follows from the formula
\begin{equation}\label{eq:14}
  \sum_{(G,m) \in \J^{\mathrm{p}}_{r,n}} \frac{(-1)^{|E(G)|}}{|\Aut(G)|}
  = (-1)^{n+1} \frac{(r+n-2)!}{r!} \cdot B_r,
\end{equation}
valid when $2r-2+n> 0$.  The case $n = 0$, $r \geq 2$ of this formula was given by Kontsevich \cite{kontsevich-formal}; see \cite[\S7.1]{gerlits-euler} for a proof.  The general case is easily deduced by induction on $n$; see  Appendix~\ref{sec:append-orbif-euler}.
\end{proof}

\section{Conclusion} \label{sec:conclusion}

In light of Lemma~\ref{lem:t-u-factor}, we may apply~(\ref{eq:13}) to the functor $\mathcal{R}: \OG_g^\mathrm{stat} \to \OG_g^{\mathrm{stat},\mathrm{red}}$ to rewrite the formula in Corollary~\ref{cor:reduction-to-static} as
\begin{align*}
  z_g & = \int_{\OG_g^\mathrm{stat}}  \sss \cdot \ttt \cdot \uuu\\
  & = \int_{\OG_g^{\mathrm{stat},\mathrm{red}}} (\mathcal{R}_* \sss) \cdot \ttt \cdot \uuu.
\end{align*}
We may now replace $\OG_g^{\mathrm{stat},\mathrm{red}}$ by the coproduct in Lemma~\ref{lem:stat-red-coprod}, and observe by Proposition~\ref{prop:u} that $\sss$ and $\ttt$ are constant functions on each $\pi_0(\OG^{\mathrm{stat},\mathrm{red}}_{(g,m,r,s,a,d)})$.  
\begin{equation*}  
  z_g   = \sum_{(k,m,r,s,a,d)} \bigg(\prod_{i=1}^s (-\mu(m/d_i))^{a_i}
         \int_{\OG^{\mathrm{stat},\mathrm{red}}_{(g,m,r,s,a,d)}} \sss \bigg)
         \bigg(P_m^{1 - r} \prod_{i = 1}^s \bigg(\frac{P_{d_i}}{P_m}\bigg)^{a_i} \bigg)
         \bigg(m^{r-1} \prod_{p | D} (1 - p^{-r})\bigg).
\end{equation*}
Combining with the formula in~\ref{lem:orbifold-euler} then finishes the proof of Theorem~\ref{thm:faber-conj}. \qed

\bigskip

We deduce a formula for the ordinary numerical (non-equivariant) top weight Euler characteristic of $\cM_{g, n}$, for large $n$.  This agrees, surprisingly, with the ``orbifold Euler characteristic'' of the category $\J_{g,n}^{\mathrm{p}}$, as computed in Proposition~\ref{prop:chi-with-mark}.  In particular, although this orbifold Euler characteristic is {\em a priori} rational, since objects in $\J_{g,n}^\circ$ have nontrivial automorphisms even when $n$ is large, it is an integer when $n > g + 1$.  We would be interested to have a more conceptual proof of these facts.

\begin{corollary} \label{cor:numerical}
The Euler characteristic of the top weight cohomology of $\cM_{g,n}$ is
\begin{equation}\label{eq:chi0c}
(-1)^{n+1} \frac{(g+n-2)!}{g!} \cdot B_g
\end{equation}
for $n>g+1$.
\end{corollary}

\begin{proof}
The Euler characteristic of the top weight cohomology of $\cM_{g,n}$ is 
\begin{equation}\label{eq:blah}
n! \cdot {\rm coeff}_{p_1^n}\,z_g.
\end{equation}
In the expression for $z_g$ in Theorem~\ref{thm:faber-conj}, the term $-\frac{B_g}{g(g-1)} \cdot P_1^{1-g}$ is the only term with a negative power of $P_1$. 
Moreover, it is easy to see that the
maximum exponent of $P_1$ equals $g+1$. Therefore for $n>g+1$,
\eqref{eq:blah} equals
$$-n! \cdot \frac{B_g}{g(g-1)} \cdot {\rm coeff}_{p_1^n} \,P_1^{1-g}.$$
Since $$P_1^{1-g}=\sum_{n\ge0} (-1)^n {\binom{g+n-2}{g-2}}p_1^n,$$
the corollary follows.
\end{proof}

\begin{remark}\label{rem:connection-to-gorsky-superficial}
In \cite{cgp-markedpoints-published} we considered a cone complex $M_{g,n}^\mathrm{trop}$ glued out of one octant $\RR_{\geq  0}^{E(G)}$ for each object $G \in \J_{g,n}$, quotiented by the action of $\Aut(G)$.  By taking these quotients in an orbifold sense instead, we obtain an orbifold incarnation of $M_{g,n}^\mathrm{trop}$, which ``keeps track of automorphisms'' of tropical curves, in a similar manner to how the orbifold $\cM_{g,n}$ keeps track of automorphisms of complex curves with marked points.  There is a closed sub-orbifold $M_{g,n}^{\mathrm{trop},w}$ made out of tropical curves for which the weighting function $w: V(G) \to \ZZ_{\geq 0}$ is not the zero function, and the number $\sum_{G \in \J_{g,n}^\mathrm{p}} (-1)^{|E(G)|}/|\Aut(G)|$ may be interpreted as $\chi^\mathrm{orb}_c(M_{g,n}^\mathrm{trop} \smallsetminus M_{g,n}^{\mathrm{trop},w})$. If written in this way, instead of evaluated in terms of Bernoulli numbers, and if $z_g$ is written as the $S_n$-equivariant compactly supported Euler characteristic of the coarse space associated to $M_{g,n}^\mathrm{trop}$, relative to the subspace $M_{g,n}^{\mathrm{trop},w}$, the formula in  Theorem~\ref{thm:faber-conj} expresses the $S_n$-equivariant (non-orbifold) Euler characteristics of the pair $(M_{g,n}^\mathrm{trop},M_{g,n}^{\mathrm{trop},w})$ in terms of the orbifold (non-equivariant) Euler characteristics of the same pair.  This is conceptually quite similar to \cite{gorsky-equivariant}, whose main result is a formula for the $S_n$-equivariant (non-orbifold) Euler characteristic of $\cM_{g,n}$ in terms of the orbifold (non-equivariant) Euler characteristic of $\cM_{g,n}$.
\end{remark}

\exnow{\label{ex:g=2}

This is an example illustrating the computation of $z_g$ from the contributions of pairs $(G,\tau)$ given by  Proposition~\ref{prop:start}, in the case $g=2$.  Note that there are only three isomorphism classes in $\J_2^\circ$, namely the theta graph, the dumbbell graph, and the figure 8 graph.

We claim that the contributions of the dumbbell $G'$ and figure 8 $G''$ sum to zero.  To see this, note that every automorphism $\tau'$ of the dumbbell maps the bridge to itself, and hence descends to an automorphism $\tau''$ of the figure 8.  The induced map is an isomorphism $\Aut G' \xrightarrow{\cong} \Aut G''$, and the contribution of $(G',\tau')$ is minus the contribution of $(G'', \tau'')$. 

It remains to compute the contribution of the theta graph $G$, with edges $1,2,3$ and vertices $v,w$. It has 12 automorphisms:

\enumnow{
\item Identity: $\phi_V = (v)(w)$, $\phi_E = (1)(2)(3).$

Contribution: 
$$-\frac{1}{12}\frac{P_1^2P_1^3}{P_1^6} = -\frac{1}{12}\frac{1}{P_1}.$$

\item Exchange 2 edges: $\phi_V = (v)(w)$, $\phi_E = (12)(3)$, 3 in this conjugacy class 

Contribution: 
$$\frac{1}{4} \frac{P_1^2 P_1P_2}{P_1^2P_2^2} = \frac{1}{4} \frac{P_1}{P_2}.$$

\item Cycle 3 edges: $\phi_V = (v)(w)$, $\phi_E = (123)$, 2 in this conjugacy class 

Contribution: $$-\frac{1}{6} \frac{P_1^2 P_3}{P_3^2}=-\frac{1}{6} \frac{P_1^2}{P_3}.$$

\item Flip: $\phi_V = (vw)$, $\phi_E = (1)(2)(3)$, 1 in this conjugacy class 

Contribution: 
$$-\frac{1}{12} \frac{P_2P_1^3}{P_2^3}=-\frac{1}{12} \frac{P_1^3}{P_2^2}.$$

\item Exchange 2 edges and flip: $\phi_V = (vw)$, $\phi_E = (12)(3)$, 3 in this conjugacy class 

Contribution: 
$$\frac{1}{4} \frac{P_2P_1P_2}{P_2^3}=\frac{1}{4} \frac{P_1}{P_2}.$$

\item Cycle 3 edges and flip: $\phi_V = (vw)$, $\phi_E = (123)$, 2 in this conjugacy class 

Contribution: 
$$-\frac{1}{6} \frac{P_2P_3}{P_6}.$$

}

The total
$$z_2 = -\frac{1}{12}\frac{1}{P_1}+\frac{1}{2}\frac{P_1}{P_2}-\frac{1}{6}\frac{P_1^2}{P_3}-\frac{1}{12}\frac{P_1^3}{P_2^2}-\frac{1}{6}\frac{P_2P_3}{P_6}$$
agrees with Theorem~\ref{thm:faber-conj}.
}

\section{Genus 0 and 1}
\label{sec:genus-0-1}

We recall (see \cite[Lemma~3.6]{sundaram-plethysm}) the following formula for Frobenius characteristics of induced representations of $S_n$.
\begin{lemma}\label{lem:sundaram}
Let $H\le S_n$ and $W$ a representation of $H$.  Then the Frobenius characteristic of $\mathrm{Ind}^{S_n}_H W$  is
$$\frac{1}{|H|} \sum_{\sigma\in H} \chi_W(\sigma) \cdot \psi(\sigma).$$
\end{lemma}

\begin{proof}[Proof of Proposition~\ref{prop:genus0}]
  Recall from \cite{cgp-markedpoints-published} that $\Gr^W_{2d}H^{2d-*}(\cM_{g,n};\QQ) \cong \widetilde{H}_{*-1}(\Delta_{g,n};\QQ))$, where $d = 3g-3+n$ and $\Delta_{g,n}$ is the moduli space of tropical curves of volume one. 
  For $g = 0$ and $n \geq 4$ the space $\Delta_{0,n}$ is a shellable simplicial complex, homotopy equivalent to a wedge sum of $(n-2)!$ spheres of dimension $n-4$. (It is a point for $n = 3$, and empty for $n < 3$.)  In \cite{robinson-whitehouse-tree}, Robinson and Whitehouse show that the character of the $S_n$-action on $\widetilde{H}_{n-4}(\Delta_{0,n};\QQ) \cong \Gr^W_{2n-6}H^{n-3}(\cM_{0,n};\QQ)$ is equal to
$$\epsilon \cdot (\mathrm{Ind}^{S_n}_{S_{n\!-\!1}} \mathrm{Lie}_{n\!-\!1} - \mathrm{Lie}_{n})$$
where $\epsilon$ denotes the alternating character, and $\mathrm{Lie}_n$ is the character of the $S_n$-action on the component of the free Lie algebra on $x_1,\ldots,x_n$ spanned by words involving each generator exactly once. Recall that $\mathrm{Lie}_{n} = \mathrm{Ind}^{S_n}_{C_n} \rho_n$ for $\rho_n$ any primitive character of $C_n := \ZZ/n\ZZ$, i.e., one sending the generator $1\in \ZZ/n\ZZ$ to a primitive $n^{\text{th}}$ root of unity.  Writing 
$\mathrm{ch}$ for Frobenius characteristic,
we deduce the following formulas by applying Lemma~\ref{lem:sundaram} and using that $\mathrm{ch}(V)$ and $\mathrm{ch}(\epsilon \cdot V)$ are related by the involution of $\widehat\Lambda$ given by $p_a \mapsto (-1)^{a-1} p_a$.
\begin{eqnarray*}
\mathrm{ch}(\mathrm{Lie}_n) &=& \frac{1}{n} \sum_{d|n}  \mu(d)\cdot p_d^{n/d}\\
\mathrm{ch}(\epsilon\cdot \mathrm{Lie}_n) &=&(-\!1)^n\cdot \frac{1}{n}\sum_{d|n}  \mu(d)\cdot(-p_d)^{n/d}\\
\mathrm{ch}(\mathrm{Ind}^{S_n}_{S_{n\!-\!1}}\mathrm{Lie}_{n-1}) &=& \frac{1}{n\!-\!1} \sum_{d|(n\!-\!1)} \mu(d)\cdot p_1p_d^{(n\!-\!1)/d} \\
\mathrm{ch}(\epsilon\cdot \mathrm{Ind}^{S_n}_{S_{n\!-\!1}}\mathrm{Lie}_{n-1}) &=& (-\!1)^{n-1}\cdot\frac{1}{n\!-\!1}\sum_{d|(n\!-\!1)} \mu(d) \cdot p_1(-p_d)^{(n\!-\!1)/d} .
\end{eqnarray*}
By the Robinson--Whitehouse result we therefore get
\begin{eqnarray*}
  z_0 &=& \sum_{n\ge 3} (-1)^{n-3} \mathrm{ch} \left(
          \epsilon\cdot (\mathrm{Ind}^{S_n}_{S_{n\!-\!1}} \mathrm{Lie}_{n\!-\!1} - \mathrm{Lie}_n)
          \right)\\
&=& \sum_{n\ge 3}\left(\frac{1}{n\!-\!1}\! \sum_{d|(n\!-\!1)}  \mu(d) \cdot p_1(-\!p_d)^{(n\!-\!1)/d} +\frac{1}{n} \sum_{d|n} \mu(d) \cdot (-\!p_d)^{n/d} \right)\\
&=& \left(\sum_{d\ge 1} \sum_{i\ge 1} (1+p_1) \frac{\mu(d)}{d} \frac{(-\!p_d)^i}{i} \right) + p_1 + \frac{p_1^2-p_2}{2}.
\end{eqnarray*}
The last equality follows from switching the order of summation between $d$ and $n$ and writing $(n-1)/d$ and $n/d$ as $i \in \ZZ_{\geq 1}$; the last two terms are correction terms ensuring equality in degrees below $3$.  This expression then reduces to the one in Proposition~\ref{prop:genus0}. 
\end{proof}

\begin{proof}[Proof of Theorem~\ref{prop:genus1}]
By \cite{cgp-markedpoints-published}, for each $n\ge 3$, the space $\Delta_{1,n}$ has the homotopy type of a wedge of $(n-1)!/2$ spheres of dimension $n-1$. (It is easy to see that $\Delta_{1,1}$ and $\Delta_{1,2}$ are contractible.) Moreover, for $n\ge 3$, the homology $\widetilde{H}_{n-1}(\Delta_{1,n}; \QQ)$ as an $S_n$-representation is isomorphic to 
\[
\mathrm{Ind}_{D_n,\phi}^{S_n} \, \mathrm{Res}^{S_n}_{D_n,\rho} \, \mathrm{sgn}.
\] 
\noindent Here, $\phi\colon D_n \rightarrow S_n$ is the dihedral group of order $2n$ acting on the vertices of an $n$-gon, $\rho\colon D_n \rightarrow S_n$ is the action of the dihedral group on the edges of the $n$-gon, and $\mathrm{sgn}$ denotes the alternating representation of $S_n$.
Then we calculate $z_1$ as follows.
\begin{eqnarray*}
z_1 &=&\sum_{n\ge 3} (-1)^n \mathrm{ch} \left(\mathrm{Ind}_{D_n,\phi}^{S_n} \, \mathrm{Res}^{S_n}_{D_n,\rho} \, \mathrm{sgn}\right)\\
&=&\frac{1}{2}\sum_{n\ge3} (-1)^n\cdot \frac{1}{n}\left(\sum_{d|n} (-1)^{d+1}\phi(d)\cdot(-p_d)^{n/d}\right) -\frac{1}{2}\sum_{n=2k+1} \! (-1)^{k} p_1p_2^k\\
\vspace{.5cm}\\
&+& \frac{1}{4}\sum_{n=2k+2} (-1)^{k+1} p_1^2 p_2^k \,\, + \frac{1}{4} \sum_{n=2k+2}  (-1)^{k} p_2^{k+1}.
\end{eqnarray*}
We pause to explain each of the four terms. The first term arises from rotations $\sigma$ of a regular $n$-gon that induce a permutation on vertices of $n/d$ cycles of length $d$. There are $\phi(d)$ such rotations for each $d$ dividing $n$. The sign of $\sigma$, as a permutation on {\em edges}, is $(-1)^{(d+1)n/d}.$  (It happens that the sign of $\sigma$ as a permutation on vertices is the same, but this will not be the case in the last two calculations below.) 

The second term arises from the $n$ reflections $\sigma$ of a regular $n$-gon, for $n=2k+1\ge 3$ odd.  The sign of $\sigma$ as a permutation of edges is $(-1)^{k}$, whereas regarding $\sigma$ as a permutation of vertices yields $\psi(\sigma) = p_1p_2^k$.

The third term arises from the $n/2$ reflections $\sigma$ of a regular $n$-gon, for $n=2k+2\ge 4$ even, about an axis that passes through two vertices.  The sign of $\sigma$ as a permutation of edges is $(-1)^{k+1}$, whereas regarding $\sigma$ as a permutation of vertices yields $\psi(\sigma)=p_1^2p_2^k$.

The fourth term arises from the $n/2$ reflections $\sigma$ of a regular $n$-gon, for $n=2k+2\ge 4$ even, about an axis that passes through two sides.  The sign of $\sigma$ as a permutation of edges is $(-1)^{k}$, whereas regarding $\sigma$ as a permutation of vertices yields $\psi(\sigma)=p_2^{k+1}$.

Summing over $d$ and $k$ instead of $n$, the expression for $z_1$  becomes
$$-\frac{1}{2}\sum_{d\ge1} \frac{\phi(d)}{d}\log(1+p_d) + \left(-\frac{1}{4}-\frac{p_1}{2}-\frac{p_1^2}{4}\right)\sum_{k\ge 0}(-p_2)^k + \frac{1}{4} + p_1.$$
The terms $\frac{1}{4} + p_1$ simply arise as correction terms to ensure equality in degree below $3$.  This yields the claimed expression in Proposition~\ref{prop:genus1}.
\end{proof}

\section{Appendix: Orbifold Euler characteristics}
\label{sec:append-orbif-euler}

\propnow{\label{prop:chi-with-mark}
For each $(g,n)$ with $2g-2+n>0$,
\begin{equation}\label{eq:chi-with-mark}
\sum_{G\in \J_{g,n}^{\mathrm{p}}} \frac{(-1)^{|E(G)|}}{|\!\Aut G|} = (-1)^{n+1} \frac{(g+n-2)!}{g!} \cdot B_g.
\end{equation}
}
\begin{proof}
  Let us write $\chi^\mathrm{orb}_{g,n}=\sum_{G\in \J_{g,n}^{\mathrm{p}}} \frac{(-1)^{|E(G)|}}{|\!\Aut G|}$ for brevity.  We first treat the case $g \geq 2$, where the case $n = 0$ is due to Kontsevich \cite{kontsevich-formal}:
$$\chi^\mathrm{orb}_{g,0} = -\frac{B_g}{g(g-1)}.$$
See also \cite[\S7.1]{gerlits-euler}.

 We prove the general case by induction on $n$.  Assume $n \geq 1$.  We must show that $\chi^{\rm orb}_{g,n}  = (2-g-n)\cdot \chi^{\rm orb}_{g,n-1}$.
 
Let $A(g,n,i)$ be the set of graphs $G\in \J_{g,n}^{\circ}$ with $i$ oriented edges, together with a total ordering the edges, up to isomorphism of all this structure.  Then 
$$\chi^{\rm orb}_{g,n}  = \sum_{i} (-1)^i \cdot \frac{|A(g,n,i)|}{2^i i!}.$$
Now let $A(g,n,i)= A^-(g,n,i) \coprod A^+(g,n,i),$ where $ A^-(g,n,i)$ is the subset of isomorphism classes in which forgetting the $n^{\rm{th}}$ marked point results in an unstable graph. That is, the vertex supporting the $n^{\rm{th}}$ marked point is either 2-valent and has no other markings, or is 1-valent and has exactly one additional marking.  We will compute the sizes of $A^-(g,n,i)$ and $A^+(g,n,i)$ in terms of the sizes of $A(g,n\!-\!1,i\!-\!1)$ and $A(g,n\!-\!1,i)$, respectively.

First, we claim $$|A^+(g,n,i)| = (1-g+i)\cdot |A(g,n\!-\!1,i)|,$$
since the map $A^+(g,n,i)\to A(g,n\!-\!1,i)$ forgetting the $n^{\rm {th}}$ marked point has fibers of size $1-g+e,$ the number of vertices of a graph of genus $g$ and $e$ edges.

Second, we claim $$|A^-(g,n,i)| = ((i\!-\!1) + (n\!-\!1)) \cdot \frac{2^i i!}{2^{i-1} (i-1)!} |A(g,n\!-\!1,i\!-\!1)|.$$
Indeed, first ignore the orientations and ordering of edges: then to get a graph in $A^-(g,n,i)$ from one in $A(g,n-1,i\!-\!1)$, there are $i-1$ midpoints of edges at which the $n^{\rm{th}}$ vertex could be placed, and $n-1$ ways it could placed onto a $1$-valent vertex with precisely one other marking.  The other factor accounts for the different number of orientations and orderings of edges.

Therefore $$|A(g,n,i)| = ((i\!-\!1) + (n\!-\!1))\cdot \frac{2^i i!}{2^{i-1} (i\!-\!1)!} |A(g,n\!-\!1,i\!-\!1)| + (1\!-\!g\!+\!i) |A(g,n\!-\!1,i)|.$$

Now multiply by $(-1)^i/(2^i i!)$ and sum over all $e$, reindexing the first of the two right hand by replacing $i-1$ by $i$.  This shows that $\chi^{\rm orb}_{g,n}  = (2-g-n)\cdot \chi^{\rm orb}_{g,n-1}$, as required.

The cases $g = 0$ and $g = 1$ are proved by the same method, except the induction must start with $n = 3$ and $n=1$, respectively.  But that can be done by inspection: up to equivalence $\J_{0,3}^p$ has just the ``tripod'' as  single object, with only the identity automorphism, while $\J_{1,1}^p$ has just the ``marked loop'' as its single object, with automorphism group $\ZZ/2\ZZ$.  This fits with $B_0 = 1$ and $B_1 = -\frac12$.
\end{proof}

\bibliographystyle{amsalpha}
\bibliography{my}

\end{document}